\documentclass[reqno]{amsart}
\usepackage{amssymb,amsmath,ascmac,amsfonts,amsthm}
\usepackage{mathrsfs} 
\usepackage[dvipdfmx]{graphicx}
\usepackage{xcolor}
\usepackage{bm}
\usepackage{url}
\usepackage{mathtools}
\mathtoolsset{showonlyrefs=true}
\definecolor{labelkey}{rgb}{0.9451,0.2706,0.4941}

\makeatletter
       \def\@makefnmark{%
               \leavevmode
               \raise.9ex\hbox{\check@mathfonts
                       \fontsize\sf@size\z@\normalfont%
                               \@thefnmark}%
       }
\makeatother

\newcommand{\ep}{\varepsilon}
\newcommand{\vnorm}[1]{|\hspace{-0.3mm}|\hspace{-0.3mm}|#1|\hspace{-0.3mm}|\hspace{-0.3mm}|}
\newcommand{\dual}[1]{\langle{#1}\rangle}
\newcommand{\dualV}[1]{\dual{#1}_{V',V}}
\newcommand{\dualW}[1]{\dual{#1}_{W',W}}

\newfont{\bg}{cmr10 scaled\magstep2}

\newcommand{\bigzerou}{\smash{\lower1.7ex\hbox{\bg 0}}}
\newcommand{\dst}[1]{\operatorname{dist}#1}
\newcommand{\ball}[2]{\operatorname{B}_{#1}(#2)}

\newcommand{\nullsp}{\mathcal{N}}
\newcommand{\ran}{\mathcal{R}}
\newcommand{\domain}{\mathcal{D}}

\newcommand\defeqn{\stackrel{\textrm{\scriptsize def.}}{=}}

\definecolor{MyBlue}{RGB}{15,82,188} 
\definecolor{mygreen}{cmyk}{0.64,0.00,0.95,0.40}

\begin{document}

\theoremstyle{plain}
\newtheorem{theorem}{Theorem}
\newtheorem{corollary}[theorem]{Corollary}
\newtheorem{lemma}[theorem]{Lemma}
\theoremstyle{definition}
\newtheorem{dfn}[theorem]{Definition}
\newtheorem{example}[theorem]{Example}
\newtheorem{remark}[theorem]{Remark}
\newtheorem{fact}[theorem]{Fact}

\title[BNB theorem and Kato's minimum modulus]{Notes on the Banach--Ne{\v c}as--Babu{\v s}ka theorem and
Kato's minimum modulus of operators}
\author[N. Saito]{Norikazu Saito}
\address{Graduate School of Mathematical Sciences, The University of
Tokyo, Komaba 3-8-1, Meguro, Tokyo 153-8914, Japan}
\email{norikazu@g.ecc.u-tokyo.ac.jp}
\urladdr{http://www.infsup.jp/saito/index-e.html}
\date{\today}

\maketitle

\begin{abstract}
 This note was prepared for a lecture given at Kyoto University (RIMS Workshop: ``The State of the Art in Numerical Analysis:
 Theory, Methods, and Applications'', November 8--10, 2017). 
That lecture described the variational analysis of the
 discontinuous Galerkin time-stepping method for parabolic equations
 based on an earlier paper by the author \cite{sai17}.
I also presented the Banach--Ne{\v
 c}as--Babu{\v s}ka (BNB) Theorem or the Babu{\v s}ka--Lax--Milgram
 (BLM) Theorem as the key theorem of our analysis.
 For proof of the BNB theorem, it is useful to
 introduce the minimum modulus of operators by T. Kato.
This note presents a review of the proofs of Closed Range Theorem and BNB
 Theorem following the idea of Kato. Moreover, I present an application
 to BNB theorem to parabolic equations. The
 well-posedness is proved by BNB theorem. 
 This note is not an original research paper. It includes no new results. 
This is a revised manuscript and several incorrect descriptions in the original version are fixed.     
\end{abstract}



\addtocounter{section}{-1}
\section{Notation}
\label{sec:00}

All functions and function spaces in this note are real-valued. 

Letting $X$ be a Banach space with the norm denoted as $\|\cdot\|_X$, then
the dual space of $X$, say, the set of all linear bounded functional
defined on $X$ is denoted by $X'$. For $\varphi\in X'$, we write
$\varphi(x)=\dual{\varphi,x}_{X',X}=\dual{x,\varphi}_{X',X}$ and call it the duality
pairing between $X'$ and $X$.
The norm of $X'$ is defined as 
\[
 \|\varphi\|_{X'}\defeqn \sup_{x\in X}\frac{\dual{\varphi,x}_{X',X}}{\|x\|_X}\qquad
 (\varphi\in X').
\]
It is well known that $X'$ forms a Banach space equipped with the norm
$\|\cdot\|_{X'}$. 

Letting $Y$ be a (possibly another) Banach space, the set of all bounded
bilinear forms on $X\times Y$ is designated as $\mathcal{B}(X,Y)$. That
is, if $b\in\mathcal{B}(X,Y)$, then $b(\cdot,y)$ is a linear functional on $X$ for a fixed $y\in Y$, then
  $b(x,\cdot)$ is a linear functional on $Y$ for a fixed $x\in X$, and
 \[
  \|b\|\defeqn \sup_{x\in X,y\in Y}\frac{b(x,y)}{\|x\|_X\|y\|_Y}<\infty.
 \]

 For a subset $M$ of $X$, we set
  \begin{equation*}
    M^\bot=\{f\in X'\mid \dual{f,x}_{X',X}=0\ (\forall x\in M)\},
  \end{equation*}
  which is called the {annihilator} of $M$.  
  The space $M^\bot$ is a closed subspace of $X'$.
  We write 
  \begin{equation*}
   \dst_{X}(x,M)=\inf_{z\in M}\|x-z\|_X\qquad (x\in X).
  \end{equation*}
  
Let $T$ be an operator from $X$ into $Y$ with its domain
$\mathcal{D}(T)\subset X$. $\nullsp (T)=\{x\in \mathcal{D}(T)\mid
Tx=0\}$ is the null space (kernel)
of $T$ and $\ran (T)=\{Tx\in Y\mid x\in \mathcal{D}(T)\}$ is the range
(image) of $T$. $\nullsp (T)$ is a closed subspace of $X$ and $\ran(T)$ is a
subspace of $Y$. The set of all bounded linear operators of $X\to Y$ with
their domain $X$ is
denoted by $\mathcal{L}(X,Y)$: if $T\in \mathcal{L}(X,Y)$, then $T$
is a linear operator of $X\to Y$ with $\mathcal{D}(T)=X$ and
\[
 \|T\|_{X,Y}\defeqn\sup_{x\in X}\frac{\|Tx\|_Y}{\|x\|_X}<\infty,
\]
which is called the operator norm of $T$.

\section{Introduction}
 \label{sec:0}

\subsection{Banach--Ne{\v c}as--Babu{\v s}ka Theorem}
\label{sec:1.1}
The present note presents specific examination of the following theorem called
the
Banach--Ne{\v c}as--Babu{\v s}ka (BNB) Theorem or the Babu{\v s}ka--Lax--Milgram
(BLM) Theorem. 

 \begin{theorem}
  \label{th:bnb}
Letting $V$ be a Banach space and letting $W$ be a reflexive Banach space, 
then, for any $a\in \mathcal{B}(V,W)$, 
  the following \textup{(i)}--\textup{(iii)} are equivalent.
\begin{itemize}
 \item[\textup{(i)}] For any $L\in W'$, there exists a unique $u\in V$
	      such that   
\begin{equation}
		 \label{eq:1}
a(u,w)=\dualW{L,w}\qquad (\forall w\in W).
		\end{equation}
 \item[\textup{(ii)}] 
	      \begin{subequations} 
 \label{eq:bnb}
\begin{gather}
\exists \beta>0,\quad \inf_{v\in V}\sup_{w\in W}\frac{a(v,w)}{\|v\|_{V}\|w\|_{W}}= \beta ; \label{eq:bnb1}\\
w\in W,\quad (\forall v\in V,\ a(v,w)=0)\quad
 \Longrightarrow \quad (w=0). \label{eq:bnb2}
\end{gather}
	      \end{subequations}
 \item[\textup{(iii)}] 
\begin{equation}
 \label{eq:bnb3}
  \exists \beta_1,\beta_2>0,\quad
  \inf_{v\in V}\sup_{w\in W}\frac{a(v,w)}{\|v\|_{V}\|w\|_{W}}= \beta_1,\  
  \inf_{w\in W}\sup_{v\in V}\frac{a(v,w)}{\|v\|_{V}\|w\|_{W}}= \beta_2.
\end{equation}
\end{itemize}
 \end{theorem}

   \begin{remark}
If \eqref{eq:bnb3} is satisfied, then we have $\beta_1=\beta_2$. Moreover, the value of $\beta$ in \eqref{eq:bnb1} agrees with $\beta_1=\beta_2$ in \eqref{eq:bnb3}.
  \label{rem:bnb0}
   \end{remark}

 \begin{remark}
Condition \eqref{eq:bnb1} is expressed equivalently as
\[
\exists \beta>0,\quad \sup_{w\in W}\frac{a(v,w)}{\|w\|_{W}}\ge
  \beta \|v\|_V  \quad (\forall v\in V).
\]  
Usually, \eqref{eq:bnb1} is called the \emph{Bab{\v s}ka--Brezzi} condition or the \emph{inf--sup} condition. 
  \label{rem:bnb00}
 \end{remark}

  \begin{remark}
Condition \eqref{eq:bnb2} is expressed equivalently as
\[
\sup_{v\in V}|a(v,w)|>0\quad (\forall w\in W,w\ne 0).
\]  
  \label{rem:bnb00b}
 \end{remark}

 \begin{remark}
The solution $u\in V$ of \eqref{eq:1}
  satisfies
  \[
   \|u\|_V \le \frac{1}{\beta}\|L\|_{W'}
  \]
in view of \eqref{eq:bnb1} and Remark \ref{rem:bnb00}.   
  \label{rem:bnb1}
 \end{remark}

   \begin{remark}
If $V$ and $W$ are finite-dimensional and $\dim V=\dim W$, then \eqref{eq:bnb1}
    implies \eqref{eq:bnb2}. See \cite[Proposition 2.21]{eg04}.
  \label{rem:fd}
   \end{remark}

Theorem \ref{th:bnb} might be understood as a generalization of the following
fundamental result, called the {Lax--Milgram Theorem}.
 
 \begin{theorem}
  \label{th:lm}
  Letting $a\in \mathcal{B}(V,V)$, where $V$ is a Hilbert space, then we assume
  that a positive constant $\alpha$ exists such that 
  \begin{equation}
   a(v,v)\ge \alpha \|v\|_V^2.
		 \label{eq:lm1}
  \end{equation}
  Then, for any $L\in V'$, there exists a unique $u\in V$ such that
  \begin{equation*}
a(u,w)=\dualV{L,w}\qquad (\forall w\in V).
\end{equation*}
 \end{theorem}

 This theorem was presented in \cite[theorem 2.1]{lm54}; the special
 case was presented earlier in \cite{vis51}. It is
 interesting that the main aim of \cite{lm54} is to resolve higher order
 parabolic equations by Hille--Yosida's semigroup theory. It is
 described in \cite{lm54} that  
  \begin{quote}
\emph{The following theorem is a mild generalization of the Fr{\' e}chet--Riesz
 Theorem on the representation of bounded linear functionals in Hilbert
 space. }[page 168]
 \end{quote}

%

The condition \eqref{eq:lm1} is usually called the coercivity
condition. If $W=V$, then \eqref{eq:lm1} implies \eqref{eq:bnb3};
Theorem \ref{th:lm} is a corollary of Theorem \ref{th:bnb}.   

Theorem \ref{th:bnb} has a long history.
 \begin{itemize}
  \item In 1962, Ne{\v c}as \cite[Th{\' e}or{\`e}me 3.1]{nec62} proved that
      part ``(iii) $\Rightarrow$ (i)'' for the Hilbert case (i.e., the case where both $V$ and $W$
      are Hilbert spaces) as a simple generalization of the Lax--Milgram
     theorem. Ne{\v c}as described that\footnote{In quotations below, we have adapted reference numbers
      for the list of references of this paper. }
\begin{quote}
\emph{Consid{\' e}rant les espaces complexes et les op{\' e}rateurs diff{\'
 e}rentiels elliptiques, le th{\' e}or{\` e}me de P. D. Lax and
 A. Milgram (cf. p. ex. L. Nirenberg \cite{nir55}) para{\^ i}t {\^ e}tre tr{\`
 e}s utile pour la m{\' e}thode variationnelle
 d'abord nous en signlons une g{\' e}n{\' e}ralisation facile.} [page 318]
\end{quote}
He also described that (see \cite[Th{\' e}or{\`e}me 3.2]{nec62})
\eqref{eq:bnb1} and
\[
  \ran(A)\mbox{ is dense in }W'
\]
implies (i) for the Hilbert case, where $A$ denotes the associating
	operator with $a(\cdot,\cdot)$; see \eqref{eq:a1} for the
	definition. Later, in 1967, Ne{\v c}as \cite[Th{\' e}or{\`e}me 6-3.1]{nec67} proved
      that \eqref{eq:bnb1} and
\[
  \exists c>0,\quad
  \sup_{v\in V}\frac{a(v,w)}{\|v\|_{V}}\ge c \|w\|_{Z}\quad (w\in W)
\]
implies (i) for the Hilbert case, where $Z$ denotes a Banach space such
      that $W\subset Z$ (algebraically and topologically). See also
	\cite{nec12}. I infer that
      Ne{\v c}as noticed the part ``(ii) $\Rightarrow$
	(i)''.
\item In 1968, Hayden \cite[Theorem 1]{hay68} proved that 
\[
  \eqref{eq:bnb1}\mbox{ and  }\nullsp(A)=\nullsp (A')=\{0\} \quad \Leftrightarrow\quad \mbox{(i) }
\]
for the Banach case, where $A'$ denotes the dual operator of
      $A$; see \eqref{eq:a2} for the
	definition. 
     
\item In 1971, Babu{\v s}ka \cite[theorem 2.1]{bab71} stated the part ``(iii)
      $\Rightarrow$ (i)'' for the Hilbert case. Babu{\v s}ka
      described that \footnote{However, I was unable to find where proof of
      the theorem was given in \cite{nir55}.}
      \begin{quote}
       \emph{The proof is adapted from Ne{\v c}as \cite{nec62} and
       Nirenberg \cite{nir55}. We present this proof because we shall
       use a portion of it for proof of the next theorem.} [page 323] 
      \end{quote}
      Later, Babu{\v s}ka--Aziz \cite[Theorem 5.2.1]{ba72} stated in 1972 the part ``(ii) $\Rightarrow$
     (i)'' for the Hilbert case. It is
      described that
      \begin{quote}
       \emph{This theorem is a generalization of the well known
       Lax--Milgram theorem. The theorem might be generalized easily to
       the case where $H_1$ and $H_2$ are reflexive Banach spaces. The
       method proof is an adaptation from \cite{nec62} and \cite{nir55}
       (see also Necas \cite{nec67}, p.294).} [page 116] 
      \end{quote}
  \item In 1972, Simader \cite[Theorem 5.4]{sim72} presented the part ``(iii) $\Rightarrow$ (i)'' for $V=W^{m,p}_0(\Omega)$ and
	$W=W^{m,q}_0(\Omega)$, where $\Omega\subset \mathbb{R}^n$ is a bounded smooth domain, $1<p,q<\infty$,
	$\frac{1}{p}+\frac{1}{q}=1$ and $1\le m\in\mathbb{Z}$. The proof
	could be applied to the general reflexive Banach spaces $V$ and
	$W$. It is noteworthy that \cite{sim72} is essentially an
	English translation of his dissertation in 1968. 
  \item In 1974, Brezzi \cite[Corollary 0.1]{bre74} proved the part
      ``(i)$\Leftrightarrow$(iii)''for the Hilbert case. 
It is
      described that
      \begin{quote}
       \emph{the results contained in theorem 0.1 and in corollary 0.1
       are of classical type and that they might not be new. For instance part
       I)$\Rightarrow$III) of corollary 0.1 was used by Babu{\v s}ka
       \cite{bab71}. } [page 132]
      \end{quote}
\item In 1989, Ro{\c s}ca \cite[Theorem 3]{ros89} proved the part
      ``(i)$\Leftrightarrow$(ii)'' for the Banach case and called it the
      Babu{\v s}ka--Lax--Milgram theorem\footnote{In the article ``Babuska--Lax--Milgram theorem'' in \emph{Encyclopedia of
      Mathematics} (\url{http://www.encyclopediaofmath.org/}), the part
      ``(i)$\Leftrightarrow$(ii)'' of Theorem \ref{th:bnb} is called the Babuska--Lax--Milgram
	Theorem. (This article was written by I. Ro{\c s}ca.)
      }.
  \item
       In 2002, Ern and Guermond presented the part
      ``(i)$\Leftrightarrow$(ii)'' as \emph{Theorem of {Ne{\v
      c}as}} in their monograph
      \cite[\S 3.2]{eg02}. Later, they named the part
      ``(i)$\Leftrightarrow$(ii)'' the {Banach--Ne{\v
      c}as--Babu{\v s}ka Theorem} in an expanded version
       of \cite{eg02}; see \cite[\S 2.1]{eg04}. 
       It is described in \cite{eg04} that 
	\begin{quote}
\emph{The BNB Theorem plays a fundamental role in this book. Although it is by
no means standard, we have adopted the terminology ``BNB Theorem'' because
the result is presented in the form below was first stated by Ne{\v c}as in
1962 \cite{nec62} and popularized by Babuska in 1972 in the context of finite
element methods \cite[p. 112]{ba72}. From a functional analysis perspective,
this theorem is a rephrasing of two fundamental results by Banach: the Closed
Range Theorem and the Open Mapping Theorem.} [page 84]
	\end{quote}	      
  \item I could find no explicit reference to the part ``(ii) $\Leftrightarrow$
	(iii)''. However, it is known among specialists.  
 \end{itemize}

As for the naming of Theorem \ref{th:bnb}, I follow conventions in \cite{eg04}.

\subsection{Operator version of Theorem \ref{th:bnb}}
\label{sec:1.2}
To elucidate Theorem \ref{th:bnb} more deeply, it is useful to reformulate it using
operators. Below, supposing that $V$, $W$, and $a$ are those described in
Theorem \ref{th:bnb}, unless otherwise stated explicitly, then we introduce $A\in\mathcal{L}(V,W')$ as
\begin{equation}
 \label{eq:a1}
  a(v,w)=\dualW{Av,w}\qquad (v\in V,~w\in W).
\end{equation}
Then, (i) of Theorem \ref{th:bnb} is interpreted as ``the operator
$A:V\to W'$ is bijective''.
The \emph{dual (adjoint) operator} $A':W\to V'$ of $A$ is defined as 
\begin{equation}
 \label{eq:a2}
  \dualW{Av,w}=\dualV{v,A'w}\qquad (v\in V,~w\in W).
\end{equation}
Then we have $A'\in \mathcal{L}(W,V')$.  
We introduce 
\begin{equation}
\label{eq:2.2w2}
 \mu(A) =\inf_{v\in V} \frac{\|Av\|_{W'}}{\|{v}\|_{{V}}}\quad\mbox{and}\quad 
 \mu(A') =\inf_{w\in W} \frac{\|A'w\|_{V'}}{\|{w}\|_{{W}}}, 
 \end{equation}
which we will call the \emph{minimum modulus} of operators (see
Definition \ref{dfn:1}). 

Because 
 \begin{align*}
\mu(A)&=\inf_{v\in V}\frac{\|Av\|_{W'}}{\|v\|_V}=
\inf_{v\in V}\frac{1}{\|v\|_{V}}\sup_{w\in W}\frac{\dualW{Av,w}}{\|w\|_{W}}=
\inf_{v\in V}\sup_{w\in W}\frac{a(v,w)}{\|v\|_{V}\|w\|_{W}},\\
  \mu(A')&=
  \inf_{w\in W}\frac{\|A'w\|_{V'}}{\|w\|_{W}}
  \inf_{w\in W}\frac{1}{\|w\|_W}\sup_{v\in V}\frac{\dualW{v,A'w}}{\|v\|_{V}}=
  \inf_{w\in W}\sup_{v\in V}\frac{a(v,w)}{\|v\|_{V}\|w\|_{W}},
 \end{align*}
 we have
\[
 \eqref{eq:bnb1}\ \Leftrightarrow \ \mu(A)>0
,\qquad \eqref{eq:bnb3}\ \Leftrightarrow \ \mu(A)=\mu(A')>0
\] 
 
 Moreover,
\[
\eqref{eq:bnb2}
\ \Leftrightarrow \ \nullsp (A')=\{0\}.
\]

Consequently, Theorem \ref{th:bnb} is equivalent to the following
theorem in view of \eqref{eq:a1}.

\begin{theorem}
 \label{th:5.1}
Letting $V$ be a Banach space and letting $W$ be a reflexive Banach space,
then for any $A\in\mathcal{L}(V,W')$, the following \textup{(i)}--\textup{(iii)} are equivalent: 
\begin{itemize}
 \item[\textup{(i)}] $A$ is a bijective operator of $V\to W'$;
 \item[\textup{(ii)}] $\mu(A)>0$ and $\nullsp(A')=\{0\}$; and
 \item[\textup{(iii)}] $\mu(A)=\mu(A')>0$.
\end{itemize}
In those expressions, $A'\in\mathcal{L}(W,V')$ denotes the dual operator of $A$ defined
 as \eqref{eq:a2}.
\end{theorem}

As explained clearly in \cite[\S A.2]{eg04}, the proof of Theorem
\ref{th:5.1} is an application of 
\begin{itemize}
 \item Open Mapping Theorem (or Closed Graph Theorem),
 \item Closed Range Theorem.
\end{itemize}
In fact, in view of Open Mapping Theorem (or Closed Graph Theorem), it
can be shown that
\begin{equation}
 \mu(A)>0\quad \Leftrightarrow\quad \nullsp(A)=\{0\}\mbox{ and
  }\ran(A)\mbox{ is closed}.
 \label{eq:f1}
\end{equation}
Then, combining this with Closed Range Theorem, we can prove Theorem
\ref{th:5.1}. Particularly if $\nullsp(A)=\{0\}$ and
$\nullsp(A')=\{0\}$, then $A^{-1}$ and $(A')^{-1}$ exist and
 \[
  \mu(A)=\|A^{-1}\|_{W',V},\qquad \mu(A')=\|(A')^{-1}\|_{V',W}.
 \] 
Therefore, $\mu(A)=\mu(A')$ is nothing but the standard fact of
 \[
  \|A^{-1}\|_{W',V}=\|(A')^{-1}\|_{V',W}.
 \]

\subsection{Remarks on Closed Range Theorem}
\label{sec:1.3}
In standard textbooks of numerical analysis, we use Closed Range Theorem
without proof. The proof is left as a \emph{black box}. However, in my
opinion, it is worth knowing how to prove Closed Range Theorem for
researchers of numerical analysis. 
I would like to offer an approach. I recommend introducing the following quantities:
\begin{equation}
\label{eq:2.2z2}
 \gamma(A)=
\inf_{v\in V}
\frac{\|Av\|_{W'}}{\dst_V(v,\nullsp(A))}\quad\mbox{and}\quad 
 \gamma(A')=
\inf_{w\in W} \frac{\|A'w\|_{V'}}{\dst_W(w,\nullsp(A'))}
\end{equation}
instead of $\mu(A)$ and $\mu(A')$.
Following Kato
 \cite{kat95}, we call this quantity
 $\gamma(A)$ the \emph{reduced minimum modulus} of $A$ (see Definition \ref{dfn:1}). 
 Indeed, we can prove: 
 \begin{itemize}
  \item 
  $\ran(A)$ is closed $\Leftrightarrow$ $\gamma(A)>0$ (Theorem \ref{th:2.1} below);
\item $\gamma(A)=\gamma(A')$ (Theorem \ref{th:2.2} below).  
 \end{itemize}
Particularly $\ran(A)$ is closed if and only if $\ran(A')$ is
closed. This is a part of Closed Range Theorem. These are classical
results by Kato \cite{kat58}. The main objective of \cite{kat58} is to
develop the perturbation theory for eigenvalue problems of linear
operators.
To accomplish this main objective, Kato studied $\gamma(A)$ and gave the
proof of Closed Range Theorem for closed (possibly unbounded)
operators. (The original theorem by S. Banach was formulated and proved
for bounded operators; see \cite[Theorems X.8, X.9]{ban87}.)       

Kato's proof of \cite{kat58} was later generalized in \cite{kat95}. A
simple explanation can be found in \cite{bre11}.

It is noteworthy that the introduction of $\gamma(A)$ was not originally Kato's
idea. Many researchers have introduced the same quantity. However,
Kato realized the importance of this quantity and developed his theory
using it as a key tool. R.~G. Bartle stated in \texttt{Mathematical
Review} that 
\begin{quote}
\emph{
The author introduces a constant $\gamma(A)$, called the lower bound of
 $A$, which is defined to be the supremum of all numbers $\gamma\geq 0$
 such that $\|Ax\|\geq\gamma\|\tilde x\|$, $x\in D(A)$, where $\tilde x$
 is the coset $x+N(A)$ and $\|\tilde x\|$ denotes the usual factor space
 norm in $X/N(A)$. Others have considered this constant before [see the
 reviewer's note, Ann. Acad. Sci. Fenn. Ser A. I. no. 257 (1958);
 MR0104172], but this reviewer is not aware of any previous systematic
 use of $\gamma(A)$.
 }[MR0107819]
\end{quote}

I believe that Kato's proof includes an idea full of suggestion for the
study of numerical analysis and that it is worthy of study for
researchers of numerical analysis.

\subsection{Application of Theorem \ref{th:bnb}}
\label{sec:1.5}
Ne{\v c}as originally established Theorem \ref{th:bnb}, the part
``(iii)$\Rightarrow$(i)'' to deduce the well-posedness (the unique
existence of a solution with a priori estimate) of
higher-order elliptic equations in weighted Sobolev spaces.   
However, Theorem \ref{th:bnb} plays a crucial role in the
theory of the finite element method.
Pioneering work was done for error analysis of elliptic problems (see
\cite{bab71}, \cite{ba72}). Moreover, active applications for the mixed finite
element method are well-known:
see \cite{bf91}, \cite{bbf13} and \cite{eg04} for
systematic study. Another important application is the well-posedness of
parabolic equations (see \cite[\S 6]{eg04} for
example). Although this later application is apparently unfamiliar, it is
actually useful for studying the discontinuous Galerkin time-stepping method, as reported recently in \cite{sai17}.

\subsection{Purpose and contents}
\label{sec:1.7}

This note has a dual purpose. 
The first is to review Kato's proof of Closed Range
Theorem using $\gamma(A)$ and to state the proof of Theorems \ref{th:bnb}
and \ref{th:5.1}. 
The second is to present the proof of the well-posedness of parabolic
equations using Theorem \ref{th:bnb}. To clarify the variational
characteristics of the method of analysis, we consider abstract
evolution equations of parabolic type, where the coefficient might
depend on the time.

\smallskip

The contents of this note are the following:  

\medskip

\begin{tabular}{lcl}
\ref{sec:00}. && Notations\\
\ref{sec:0}.  && Introduction\\
\ref{sec:1}.  && Preliminaries\\
\ref{sec:2}.  && Kato's minimum modulus of operators\\
\ref{sec:5}. && Proof of Theorems \ref{th:bnb} and \ref{th:5.1}\\
\ref{sec:7}. && {Application to evolution equations of parabolic type}\\
\ref{sec:a}. && {Proof of ``\eqref{eq:pf10} $\Rightarrow$
 \eqref{eq:pf11}''}\\
\ref{sec:b}.  && {Comments on the revised version}
\end{tabular}

\medskip

This note was prepared for the lecture given at Kyoto University (RIMS Workshop: ``The State of the Art in Numerical Analysis:
 Theory, Methods, and Applications'', November 8--10, 2017). This note
 is not an original research paper and includes no new results. This is a revised manuscript and several incorrect descriptions in the original version are fixed.


\section{Preliminaries}
\label{sec:1}

We recall two fundamental results, Closed Graph
Theorem and Hahn--Banach Theorem together with their consequences.  
Throughout this section, $X$ and $Y$ are assumed to be Banach spaces.

 
\begin{lemma}[Closed Graph Theorem]
Let $T$ be a linear operator of $X\to Y$. If $X=\domain (T)$ and  
 \begin{equation}
  \mbox{$\domain (T)$ is complete under the norm }
   \|x\|_{\domain(T)}=\|x\|_X+\|Tx\|_Y,
   \label{eq:graph} 
\end{equation}
then we have $T\in\mathcal{L}(X,Y)$. 
\label{la:cgt}
\end{lemma}

An operator $T$ satisfying \eqref{eq:graph} is called a \emph{closed operator}
   and $\|x\|_{\domain (T)}$ is called the \emph{graph norm} of $T$. Closed Graph Theorem is also described as ``a closed linear operator from $X$ into $Y$ with $\domain (T)=X$ is bounded''. 
    Because $X$ and $Y$ are Banach spaces, 
    \eqref{eq:graph} is equivalent to
     \begin{equation}
      \begin{array}{ll}
       x_n\in \domain (T),& \\
	 x_n\to x \in X\mbox{ in }X& (n\to \infty)
      \end{array}
      \quad \Longrightarrow \quad x\in \domain (T).
   \label{eq:graph2} 
\end{equation}
A bounded operator is a closed operator; \eqref{eq:graph} and
\eqref{eq:graph2} are satisfied for $T\in \mathcal{L}(X,Y)$. In fact, $\|x\|_X$ and $\|x\|_{\domain
   (T)}$ are equivalent norms of $X$, because $\|x\|_{\domain(T)}=\|x\|_X+\|Tx\|_Y\le
   (1+\|T\|_{X,Y})\|x\|_X$. Therefore, $X$ is complete under $\|x\|_{\domain
   (T)}$, which implies \eqref{eq:graph}. 

Let us consider a linear operator $T:X\to Y$ such that $\nullsp(T)=\{0\}$.
     Then, the inverse operator $T^{-1}:\ran (T)\to X$ can be defined. 
     If $T^{-1}$ is bounded, then
     \begin{equation}
 \begin{array}{l}
  y_n\in \domain (T^{-1})=\ran (T),  \\
        y_n\to y \in Y\mbox{ in }Y \  (n\to \infty)
 \end{array}\quad \Longrightarrow \quad y\in \ran (T),
   \label{eq:graph3} 
     \end{equation}
 as just mentioned above.
In other words, $\ran(T)$ is a closed set in $Y$ if $T^{-1}$ is bounded. On the other
     hand, if $\ran(T)$ is closed, \eqref{eq:graph3} is
     satisfied. Therefore, we can apply Closed Graph Theorem to
     conclude that $T^{-1}\in\mathcal{L}(\ran(T),X)$. As a result, we obtain the following lemma.

\begin{lemma}
 \label{la:closed1}
 Let $T$ be a linear operator from $X$ into $Y$ such that $\nullsp(T)=\{0\}$.
Then, we have $T^{-1}\in \mathcal{L}(\ran(T),X)$ if and only if $\ran(T)$ is closed. 
\end{lemma}
     
\begin{remark}
 I presented Lemma \ref{la:closed1} as a corollary of Closed Graph
 Theorem. However, S. Banach proved the following proposition (see
 \cite[Theorem X.10]{ban87}): if $\ran(T)$ is closed,
 then there exists a positive constant $m>0$ such that, for any $y\in
 \ran(T)$, we can take $x\in X$ satisfying $y=Tx$ and $\|x\|_X\le m
 \|y\|_Y$.   
\label{rem:c}
\end{remark}


\begin{lemma}[Hahn--Banach Theorem]
\label{th:bh}
Let $E$ be a vector space and let $p$ be a functional on $E$ such
 that
\begin{align*}
 p(\lambda x)&=\lambda p(x) && (x\in E,~\lambda>0);\\
 p(x+y)&\le p(x)+p(y) &&(x,y\in E).
\end{align*}
Suppose that $G$ is a subspace (linear subset) of $E$ and that $g$ is a
 functional on $G$ satisfying
\[
 g(x)\le p(x)\quad (x\in G).
\]
 Then, there exists a functional $\tilde{g}$ on $E$, which is called the
 extension of $g$ into $E$, such that
\[
 \tilde{g}(x)=g(x)\quad (x\in G),\qquad \tilde{g}(x)\le p(x)\quad (x\in E).
\]
\end{lemma}


 We present some useful results.
 
 \begin{lemma}
  Let $M$ be a subspace of $X$. Then, every
  $g\in M'$ admits an extension $\tilde{g}\in X'$ such that
  $\|\tilde{g}\|_{X'}=\|g\|_{M'}$. 
\label{la:bh0}
 \end{lemma}

 \begin{proof}
Apply Hahn--Banach Theorem with $p(x)=\|g\|_{G'}\|x\|_X$.  
 \end{proof}

\begin{lemma}
\label{la:a.1}
For a subspace $M$ of $X$, we have 
\begin{equation}
 \dst_{X'}(f,M^\bot)=\sup_{x\in
  M}\frac{|\dual{f,x}_{X',X}|}{\|x\|_{X}}\qquad (f\in X').
 \label{eq:a.1}
\end{equation}
\end{lemma}

\begin{proof}
 Letting $f\in X'$, and introducing the restriction $f_M=f|_M:M\to\mathbb{R}$
 of $f$ into $M$, we have $f_M\in M'$. Then, \eqref{eq:a.1} is expressed
 as $\dst_{X'}(f,M^\bot)=\|f_M\|_{M'}$.

 By Lemma \ref{la:bh0},
  an extension $g\in X'$ of $f_M$ exists such that
 $\|g\|_{X'}=\|f_M\|_{M'}$. Set $h=f-g\in X'$. Consequently, we have $h\in
 M^\bot$ because $\dual{f,x}_{X',X}=\dual{g,x}_{X',X}$ for $x\in
 M$, which implies that $\dst_{X'}(f,M^\bot)\le
 \|f-h\|_{X'}=\|g\|_{X'}=\|f_M\|_{M'}$.

 On the other hand, for any $h\in M^\bot$, we have
 $|\dual{f,x}_{X',X}|=|\dual{f-h,x}_{X',X}|\le \|f-h\|_{X'}\|x\|_X$ for
 $x\in M$. Therefore, $\|f_M\|_{M'}\le \dst_{X'}(f,M^\bot)$. 
\end{proof}

\begin{lemma}
Let $C$ be an open convex subset with $0\in C$ of $X$. Supposing that
 $x_0\in X$ and $x_0\not\in C$, then there exists a $\varphi\in X'$ such
 that $\dual{\varphi,x_0}_{X',X}=1$ and $\dual{\varphi,x}_{X',X}<1$ for $x\in C$.  
\label{la:bh1aa}
\end{lemma}

 \begin{proof}
 Setting $G=\{tx_0\mid t\in\mathbb{R}\}$, we introduce a functional $g$
  on $G$ as
 \[
  g(tx_0)=t.
 \] 
 We recall that the gauge of $C$ is given as 
 \[
 p(x)=\inf\{\alpha>0\mid \alpha^{-1}x\in C\}
 \]
and that it satisfies the following:
\begin{subequations} 
 \label{eq:g}
 \begin{align}
& \exists M>0,\ 0\le p(x)\le M\|x\|_X  && (x\in X); \label{eq:g1}\\
& C=\{x\in E\mid p(x)<1\}; &&  \label{eq:g2}\\
& p(\lambda x)= \lambda p(x)  && (x\in X,\lambda>0); \label{eq:g4}\\
& p(x+y)\le p(x)+p(y) && (x,y\in X). \label{eq:g3}
 \end{align}
\end{subequations}
  Then, it is apparent that
 \begin{equation}
  \label{eq:g8}
  g(x)\le p(x)\qquad (x\in G).
 \end{equation}
Indeed, it is trivial if $x=tx_0$ with $t\le 0$. If $x=tx_0$ with $t>0$, then
   $\alpha^{-1}tx_0\in C$ implies $\alpha>t$ because $t'x_0$ with $t'\ge
  1$ cannot belong to $C$. Therefore, \eqref{eq:g8} follows.
  According to Hahn--Banach Theorem, there exists a functional $\varphi$
  defined on $X$ such that
  $\varphi(x)=g(x)$ for $x\in G$ and  
  $\varphi(x)\le p(x)$ for $x\in X$.
Using \eqref{eq:g1}, $\varphi(x)\le p(x)\le M\|x\|_X$ and
  $-\varphi(x)=\varphi(-x)\le p(-x)\le M\|x\|_X$ for $x\in
  X$. Consequently, we have $\varphi\in X'$. 
  Moreover, we have  
  $\varphi(x_0)=\dual{\varphi,x_0}_{X',X}=1$ and $\varphi(x)=\dual{\varphi,x}_{X',X}<1$ for $x\in C$. 
 \end{proof}

We recall the proof of \eqref{eq:g}
 to emphasize that $C$ must be open and convex.
 
 \begin{proof}[Proof of \eqref{eq:g1}]
  Because $C$ is an open set, an $r>0$ exists such that
  $\mathrm{B}_X(0,r)=\{x\in X\mid \|x\|_X<r\}\subset C$. Then,
 $p(x)\le \inf\{\alpha>0\mid \alpha^{-1}x\in
  \mathrm{B}_X(0,r)\}=r^{-1}\|x\|_X$.  
 \end{proof}

  \begin{proof}[Proof of \eqref{eq:g2}]
  First, let $x\in C$. Small $\ep>0$ exists such that $(1+\ep)x\in C$
   because $C$ is open. Therefore, $p(x)\le \frac{1}{1+\ep}<1$. Conversely,
    let $p(x)<1$ for $x\in X$. Then, an $\alpha\in (0,1)$ exists such
   that $\alpha^{-1}x\in C$. Therefore, $x=\alpha
   (\alpha^{-1}x)+(1-\alpha)\cdot 0\in C$ because $C$ is convex. 
 \end{proof}

  \begin{proof}[Proof of \eqref{eq:g4}]
  $p(\lambda x)=\inf\{\beta\lambda>0\mid \beta^{-1}x\in C\}=\lambda
   p(x)$. 
 \end{proof}

  \begin{proof}[Proof of \eqref{eq:g3}]
We apply \eqref{eq:g2} and \eqref{eq:g4}. 
   Letting $\ep>0$ be arbitrary, then we have $x/(p(x)+\ep)\in C$ and
   $y/(p(y)+\ep)\in C$ because
   $p\left(x/(p(x)+\ep)\right)={p(x)}/(p(x)+\ep)<1$. Therefore, 
    for any $t\in [0,1]$,
    \[
    \frac{tx}{p(x)+\ep}+\frac{(1-t)y}{p(y)+\ep}\in C.
   \]
   Choosing $t=(p(x)+\ep)/(p(x)+p(y)+2\ep)$, we obtain
   $(x+y)/(p(x)+p(y)+2\ep)\in C$. This result implies that
   $1>p((x+y)/(p(x)+p(y)+2\ep))=p(x+y)/(p(x)+p(y)+2\ep)$, which means
   that $p(x+y)<p(x)+p(y)+\ep$. Letting $\ep\downarrow 0$, we infer
   $p(x+y)\le p(x)+p(y)$. 
 \end{proof}

 \begin{lemma}
Let $M$ be a closed convex subset with $0\in M$ of $X$. Supposing that
 $x_0\in X$ and $x_0\not\in M$, then there exists a $\varphi\in X'$ such
 that $\dual{\varphi,x_0}_{X',X}>\dual{\varphi,x}_{X',X}$ for $x\in M$.  
\label{la:bh1}
\end{lemma}

 \begin{proof}  
Because $M$ is closed, we have $d=\dst_{X}(x_0,M)>0$. Apply Lemma
 \ref{la:bh1aa} to $C=\{x\in X\mid \dst_X(x,M)<d/2\}$ which is an open convex subset
 not containing $x_0$. 
 \end{proof}
 
\begin{lemma}
Letting $M$ be a subspace of $X$ and supposing that
 $x_0\in X$ and $x_0\not\in M$ with $d=\dst_X(x_0,M)>0$, then there exists a $\varphi\in X'$ such
 that $\dual{\varphi,x_0}_{X',X}=1$, $\dual{\varphi,x}_{X',X}=0$ for
 $x\in M$ and $\|\varphi\|_{X'}\le 1/d$.   
\label{la:bh2}
\end{lemma}

\begin{remark}
 We actually have $\|\varphi\|_{X'}= 1/d$.
\end{remark}

\begin{proof}[Proof of Lemma \ref{la:bh2}]
 We introduce $M_0=\{tx_0+y \mid t\in\mathbb{R},~y\in M\}$. This $M_0$
 is a subspace of $X$.
 Writing $x\in M_0$ as $x=tx_0+y$ with
 $t\in\mathbb{R}$ and $y\in M$, we have
 \begin{equation}
  |t|\le \frac{1}{d}\|x\|_X. 
  \label{eq:g10}
 \end{equation}
In fact, $\|t^{-1}x\|_X=\|x_0+t^{-1}y\|_X\ge d$ if $t\ne 0$ whereas
 \eqref{eq:g10} is trivial if $t=0$.
At this stage, we introduce a functional $g$
  on $M_0$ by
 \[
  g(x)=t\qquad (x=tx_0+y\in M_0).
 \]
By \eqref{eq:g10}, we have $g\in M_0'$ and $\|g\|_{M_0'}\le 1/d$. This
 $g$ can be extended to $X$ preserving the bound. The extension is
 denoted by $\varphi$. Then, it is apparent that
 $\varphi(x_0)=\dual{\varphi,x_0}_{X',X}=\dual{\varphi,1\cdot x_0+0}_{X',X}=1$, 
 $\varphi(x)=\dual{\varphi,x}_{X',X}=0$ for $x=0\cdot x_0+ x\in M$, and
 $\|\varphi\|_{X'}\le 1/d$.    
\end{proof}

\begin{remark}
 Lemma \ref{la:a.1} is taken from \cite[Lemma IV.2.8]{kat95}. Lemma
 \ref{la:bh1aa} could be found in \cite[Lemma 1.3]{bre11}. 
Lemma \ref{la:bh2} is taken from \cite[Theorem III.1.22]{kat95}. 

\end{remark}

 \section{Kato's minimum modulus of operators}
 \label{sec:2}

 Letting $V$ and $W$ be Banach spaces as in \S \ref{sec:0}, and noting
 particularly that $W$ is reflexive, 
 supposing that we are given $A\in\mathcal{L}(V,W')$, then the dual operator
 $A'\in\mathcal{L}(W,V')$ of $A$ is given as 
 $\dualW{Av,w}=\dualV{v,A'w}$ for $v\in V,~w\in W$.
  
If $A$ is considered as an operator from $V$ to $W$, the reflexivity of $W$ is not necessary in the following discussion. See Remark \ref{rem:77}.

The following lemma is well known. 
  
      \begin{lemma}
       \label{la:2.3}
       We have
\begin{subequations} 
  \label{eq:2.3}
 \begin{gather}
\nullsp (A')=\ran(A)^\bot;  \label{eq:2.3a}\\
\nullsp (A)=\ran(A')^\bot;  \label{eq:2.3aa}\\
\ran (A')\subset \nullsp(A)^\bot;  \label{eq:2.3b}\\
\ran (A)\subset \nullsp(A')^\bot.  \label{eq:2.3c}
 \end{gather}
($W$ needs not to be reflexive.) 
\end{subequations}
      \end{lemma}

\begin{proof}[Proof of \eqref{eq:2.3a}] 
Let $w\in \nullsp(A')\subset W$. For any $v\in V$, we have
 $\dualW{Av,w}=\dualV{v,A'w}=0$, which gives that $w\in \ran(A)^\bot$.
 Consequently $\nullsp(A')\subset \ran(A)^\bot$. The proof
 of $\nullsp(A')\supset \ran(A)^\bot$ can be shown similarly.  
\end{proof}

\begin{proof}[Proof of \eqref{eq:2.3aa}] 
In fact, it is exactly the same as the previous proof.  
\end{proof}

\begin{proof}[Proof of \eqref{eq:2.3b}] 
Letting $f\in \ran (A')\subset V'$, where $f$ is expressed as $f=A'w$ with
 $w\in W$, then for any $v\in \nullsp(A)$, we have
 $\dualV{v,f}=\dualV{v,A'w}=\dualW{Av,w}=0$. Therefore, $f\in
 \nullsp(A)^\bot$ and 
 $\ran (A')\subset \nullsp(A)^\bot$. 
\end{proof}

\begin{proof}[Proof of \eqref{eq:2.3c}] 
It is exactly the same as the previous proof.  
\end{proof}

Relations 
$\ran (A')= \nullsp(A)^\bot$ and $\ran (A)= \nullsp(A')^\bot$
are not always true because, for example, $\nullsp(A)^\bot$ is always closed but
$\ran(A')$ need not be closed.
To derive the opposite inclusions to \eqref{eq:2.3b} and \eqref{eq:2.3c}, we require some deeper
consideration. 
We will use the quotient (factor) space
 \[
 \tilde{V}=V/{\nullsp (A)}= \{\tilde{v}=v-\nullsp (A) \mid v\in V\} 
 \]
 which is a Banach space equipped with the norm
\begin{equation}
\label{eq:2.1}
\|\tilde{v}\|_{\tilde{V}}=\inf_{g\in\nullsp(A)}\|v-g\|_V=\dst_V (v,\nullsp(A)).
\end{equation}
By consideration of this notion for $v-0\in {v}-\nullsp(A)$, we have  
 \begin{equation}
\label{eq:2.1a}
\|\tilde{v}\|_{\tilde{V}}\le \|v\|_V\qquad (v\in V).
 \end{equation}

 We introduce a linear operator $\tilde{A}:\tilde{V}\to W'$ by setting
 \begin{equation}
\label{eq:2.6}
\tilde{A}\tilde{v}=Av\qquad (\tilde{v}=v-\nullsp (A)\in \tilde{V}). 
 \end{equation}
 The operator $\tilde{A}$ is bounded and
\begin{equation}
 \label{eq:2.5}
\ran (\tilde{A})=\ran (A),\qquad \nullsp
 (\tilde{A})=\{\tilde{0}\}.
\end{equation}
Therefore, the inverse $\tilde{A}^{-1}$ exists, where $\domain
 (\tilde{A}^{-1})=\ran (\tilde{A})$. In view of Closed Graph
 Theorem (see Lemma \ref{la:closed1} and Remark \ref{rem:c}),
 $\tilde{A}^{-1}$ is bounded if and only if $\ran(\tilde{A})$ is closed.
That is, we have 
 \begin{equation}
\label{eq:2.7}
  \mbox{$\ran(\tilde{A})$ is closed} \quad \Leftrightarrow\quad
  \|\tilde{A}^{-1}\|_{\ran(\tilde{A}),\tilde{V}}=\sup_{f\in \ran(\tilde{A})}\frac{\|\tilde{A}^{-1}f\|_{\tilde{V}}}{\|f\|_{W'}}<\infty.
 \end{equation}


Motivated by the observation above, we can present the following
 definition. 
 \begin{dfn}
  The \emph{minimum modulus} of an operator $T$ from a Banach space $X$ to another
 Banach space $Y$ is defined as
\begin{equation}
\label{eq:2.2w}
 \mu(T) =\inf_{x\in X} \frac{\|Tx\|_{Y}}{\|{x}\|_{{X}}}. 
 \end{equation}
 The \emph{reduced minimum modulus} of $T$ is defined as
 \begin{equation}
\label{eq:2.2z}
 \gamma(T)=
 \inf_{x\in X} \frac{\|Tx\|_{Y}}{\|\tilde{x}\|_{\tilde{X}}}=
\inf_{x\in X} \frac{\|Tx\|_{Y}}{\dst_X(x,\nullsp(T))},
 \end{equation}
where $\tilde{X}$ denotes the quotient space $\tilde{X}=X/\nullsp (T)$. 
\label{dfn:1}
\end{dfn}

\begin{remark}
It is noteworthy that $\gamma(T)=\infty$ if and only if $Tx=0$ for all $x\in X$
 It is apparent that    
 \begin{equation}
  \label{eq:2.11a}
   \nullsp(T)=\{0\}\quad \Rightarrow \quad 
\gamma(T)=\mu(T). 
 \end{equation}
 Moreover, 
 \begin{equation}
  \label{eq:2.11b}
\mu (T)>0\quad \Rightarrow \quad    \nullsp(T)=\{0\}.
 \end{equation}
\label{rem:2.9}
\end{remark}

 \begin{remark}
The quantity $\gamma(T)$ was introduced into \cite[\S 3.2]{kat58} and
  called the \emph{lower-bound} of $T$. Actually, $\gamma(T)$ was called the reduced minimum
  modulus of $M$ in \cite[\S
  IV.5]{kat95}; it is described in \cite{kat95} that the naming follows
  \cite{gt62}, where $\mu(T)$ was defined. 
 \end{remark}

It is apparent that 
 \begin{subequations}
  \label{eq:2.2}
   \begin{align}
 \gamma(A)& =
 \inf_{v\in V} \frac{\|Av\|_{W'}}{\|\tilde{v}\|_{\tilde{V}}}=
\inf_{v\in V} \frac{\|Av\|_{W'}}{\dst_V(v,\nullsp(A))};\label{eq:2.2a}\\
 \gamma(\tilde{A})& =\gamma(A);\label{eq:2.2b}\\
 \gamma(\tilde{A})&=\|\tilde{A}^{-1}\|_{\ran(\tilde{A}),\tilde{V}}^{-1}. \label{eq:2.2c}
   \end{align}
  \end{subequations}

  
Putting \eqref{eq:2.5}, \eqref{eq:2.7}, and \eqref{eq:2.2} together, we
have the following theorem.
  
  \begin{theorem}[{\cite[Lemma 322]{kat58}}]
   \label{th:2.1}
  $\ran(A)$ is closed if and only if $\gamma(A)>0$. ($W$ needs not to be reflexive.) 
  \end{theorem}



\begin{remark}
Theorem \ref{th:2.1} might be understood as a ``quantitative version'' of the result of Banach described in Remark \ref{rem:c}.    
\end{remark}

\begin{remark}
Theorem \ref{th:2.1} could be found in \cite[Theorems 5.17.3, 5.18.2]{od96}.    
\end{remark}

\begin{example}
 \label{ex:kikuchi}
We give an example of $A$ whose range $\ran(A)$ is not a closed set. Let $V=W=L^2(I)$ with
 $I=(0,1)$. We introduce $A\in\mathcal{L}(V,W')$ by
\[
 \dualW{Av,w}=\int_0^1tv(t)w(t)~dt.
\] 
(Verify that $A$ is actually a bounded linear operator of $V\to W'$.) We
 consider $f\in W'$ defined as
\[
 \dualW{f,w}=\int_0^1w(t)~dt.
\]
Then, we have $f\not\in \ran(A)$. Indeed, if there is a $u_0\in V$ such that
 $Au_0=f$, this $u_0$ must satisfy $1-tu_0=0$ a.e.~$t\in I$. The
 ``candidate'' is given as $u_0=1/t$; however, $u_0=1/t$ cannot
 belong to $V$. Next, for $\ep>0$, we consider $f_\ep\in
 W'$ and $u_\ep\in V$ defined as
\[
 \dualW{f_\ep,w}=\int_\ep^1w(t)~dt\quad\mbox{and}\quad
	 u_\ep =\begin{cases}
		 0 & (0<t<\ep)\\
	         1/t & (\ep\le t<1).
		\end{cases}
\]
Then, we have $Au_\ep=f_\ep$ and, hence, $f_\ep\in
 \ran(A)$. Moreover, we have $f\in \overline{\ran (A)}$, because
 $\|f_\ep-f\|_{W'}\to 0$ as $\ep\to\infty$. Those imply that
 $\ran(A)\ne\overline{\ran(A)}$. Therefore, $\ran(A)$ is not closed.   
 (In the similar way, we can prove that $W'=\overline{\ran(A)}$.)
 
On the other hand, because $\nullsp (A)=\{0\}$, we estimate as 
\[
 \gamma(A)\le \lim_{\ep\to 0}\frac{\|Au_\ep\|_{W'}}{\|u_\ep\|_V}=0,
\]
which implies $\gamma(A)=0$.  
\end{example}

The following theorem plays a key role below. 

    \begin{theorem}[{\cite[Lemma 334]{kat58}}]
   \label{th:2.2}
     We have
     \[
     \gamma(A)=\gamma(A').
     \]
    Particularly $\ran(A)$ is closed if and
     only if $\ran(A')$ is closed. 
    \end{theorem}
    

   \begin{proof}
    For abbreviation, we write $\gamma=\gamma(A)$ and
    $\gamma'=\gamma(A')$.

    \noindent \emph{Step 1.} 
We prove that $\gamma'\ge \gamma$. If $\gamma=\infty$, then
    we have $Av=0$ for all $v\in V$. Therefore, $0=\dualW{Av,w}=\dualV{v,A'w}$
   for all $v\in V$ and $w\in W$, which implies $A'w=0$ for all $w\in W$.
 Consequently, $\gamma'=\infty$. Therefore, we might assume that
    $0<\gamma<\infty$, because $\gamma'\ge \gamma$ might be readily apparent if
   $\gamma=0$. 
    Letting $w\in W$, then $\ran(A)$ is closed by theorem \ref{th:2.1}. Therefore,
     we can apply Lemmas \ref{la:2.3} and \ref{la:a.1} (for $X=W'$, $M=\ran(A)$) to obtain 
 \begin{equation*}
 \dst_{W}(w,\nullsp (A'))=\sup_{f\in\ran(A)}\frac{|\dual{f,w}_{W',W}|}{\|f\|_{W'}}\qquad (w\in W).
 \end{equation*}
 This, together with \eqref{eq:2.1}, implies that 
\begin{equation}
\|\tilde{w}\|_{\tilde{W}}=\sup_{f\in \ran(A)}\frac{|\dualW{f,w}|}{\|f\|_{W'}}.
\label{eq:2.10}
 \end{equation}
Therefore, for a sufficiently small $\ep>0$,
$f\in\ran(A)$ exists such that $|\dualW{f,w}|\ge
   (1-\ep)\|f\|_{W'}\|\tilde{w}\|_{\tilde{W}}$. $f$ admits the
   representation $f=Av\in \ran(A)$ with $v\in V$. Therefore, we deduce
\begin{align*}
 |\dualW{Av,w}|
    & \ge (1-\ep)\|Av\|_{W'}\|\tilde{w}\|_{\tilde{W}}\\
    & \ge (1-\ep)\gamma \|\tilde{v}\|_{\tilde{V}}\|\tilde{w}\|_{\tilde{W}}.
\end{align*}
   Using $|\dualW{Av,w}|=|\dualV{v,A'w}|\le \|v\|_V\|A'w\|_{V'}$, we have
   \[
   \|v\|_V\|A'w\|_{V'}
 \ge (1-\ep)\gamma \|\tilde{v}\|_{\tilde{V}}\|\tilde{w}\|_{\tilde{W}}.
   \]
These inequalities remain valid if $v$ is replaced by $v-g$ for any
   $g\in\nullsp (A)$. Consequently, we have
   \[
   \|A'w\|_{V'}\inf_{g\in\nullsp (A)}\|v-g\|_V
 \ge (1-\ep) \gamma \|\tilde{v}\|_{\tilde{V}}\|\tilde{w}\|_{\tilde{W}}.
   \]
    Therefore, we
   obtain 
   \[
    \gamma'\ge (1-\ep) \gamma. 
   \]
Letting $\ep\downarrow 0$, we deduce 
    $\gamma'\ge
   \gamma$.    
 %

\smallskip
    
    \noindent \emph{Step 2.} {(a)} We prove the opposite inequality $\gamma'\le \gamma$.     
Because the inequality is trivial if $\gamma'=0$, we assume that $\gamma'>0$. 
    In general, we write $\ball{X}{r}$ to express the open ball in a
    Banach space $X$
    with center $0$ and radius $r>0$; $\ball{X}{r}=\{x\in X\mid
    \|x\|_X<r\}$. 
    The closure of
    $A\ball{V}{1}=\{Av\in W'\mid v\in \ball{V}{1}\}$ in
    $W'$ is denoted as $K=\overline{A\ball{V}{1}}$, which is a
    convex closed subset in $W'$ with $0\in K$.

\smallskip
    
\noindent {(b)} We show that
\begin{equation}
\ball{W'}{\gamma'}\subset K=\overline{A\ball{V}{1}}.
 \label{eq:pf10}
\end{equation}
    To this purpose, we prove that
\begin{equation}
      f_0\in \ran(A),\ f_0\not\in K \quad \Rightarrow\quad
\|f_0\|_{W'}\ge \gamma'   .
    \label{eq:pf1}
   \end{equation}
In view of
    Lemma \ref{la:bh1}, there exists an $\eta\in (W')'$ such that
\[
 \dual{\eta,f}_{(W')',W'}<\dual{\eta,f_0}_{(W')',W'}\qquad (f\in
    K).
\]
Because $W$ is reflexive, there exists a $w\in W$ such that
\[
     \dualW{f,w}<\dualW{f_0,w}\qquad (f\in
    K). 
\]
By considering $-f$ instead of $f$, we have     
\begin{equation*}
    |\dualW{f,w}|<\dualW{f_0,w}<|\dualW{f_0,w}|\qquad (f\in
    K). 
\end{equation*}
Letting $0\ne v\in V$ and $0<\ep<1$ and setting
    $\hat{v}=(1-\ep)v/\|v\|_V\in \ball{V}{1}$, then by   
    substituting $f=A\hat{v}$, we obtain
\[
    (1-\ep)\frac{|\dualW{Av,w}|}{\|v\|_V}
    =(1-\ep)\frac{|\dualV{v,A'w}|}{\|v\|_V}
    \le |\dualW{f_0,w}|.
\]
Consequently,
\[
    (1-\ep)\sup_{v\in V}\frac{|\dualV{v,A'w}|}{\|v\|_V}
    \le |\dualW{f_0,w}|.
\]
By letting $\ep\downarrow 0$, 
\begin{equation}
    \|A'w\|_{V'}\le |\dualW{f_0,w}|.
\label{eq:161}
\end{equation}
We can apply \eqref{eq:2.10} to obtain 
\begin{equation*}
    \|A'w\|_{V'}\le \|\tilde{w}\|_{\tilde{W}}\|f_0\|_{W'}.
\end{equation*}
We know that $\|A'w\|_{V'}\ge \gamma'
    \|\tilde{w}\|_{\tilde{W}}$ for any $w\in W$. Combining these, we
    have $\|f_0\|_{W'}\ge \gamma'$, which completes the proof of
    \eqref{eq:pf1}. 

\smallskip
    
\noindent {(c)} The inclusion \eqref{eq:pf10} implies that     
\begin{equation}
\ball{W'}{\gamma'}\subset A\ball{V}{1}.
 \label{eq:pf11}
\end{equation}
This is verified by a standard argument; we will mention the detail in  Appendix \ref{sec:a}.

At this stage, letting $0\ne v\in V$ and letting $0<\ep<1$, we set
    $v^*=(1-\ep)\gamma'v/\|Av\|_{W'}$. Then, because
    $\|Av^*\|_{W'}=(1-\ep)\gamma'<\gamma'$, we have $Av^*\in A\ball{X}{1}$. This implies
    that there exists a ${v}^{\#}\in \ball{V}{1}$ satisfying
    $A{v}^{\#}=Av^*$ and ${v}^{\#}=v^*-g$ for any $g\in \nullsp(A)$. We
    have
\[
1>\|{v}^{\#}\|_V=\frac{(1-\ep)\gamma'}{\|Av\|_{W'}} \|v-\alpha g\|_V,   
\]    
where $\alpha=\|Av\|_{W'}/((1-\ep)\gamma')$. This gives that
    \[
     \|Av\|_{W'}> (1-\ep)\gamma'\|v-\alpha g\|_{V}\ge (1-\ep)\gamma'\|\tilde{v}\|_V.
    \]
    Because $\ep$ is arbitrary, we infer
    $\gamma'\|\tilde{v}\|_{\tilde{V}}\le \|Av\|_{W'}$, which implies that
    $\gamma\ge \gamma'$. This completes the proof of Theorem \ref{th:2.2}.     
\end{proof}

Using this theorem, we can prove the following results. 

\begin{theorem}
       \label{th:2.4a}
       $\ran (A)\supset \nullsp(A')^\bot$ if $\ran(A')$ is closed. 
\end{theorem}

    \begin{proof}
     Letting $f\in \nullsp(A')^\bot$, then
we prove $f\in \ran (A)$
    by presenting a contradiction: assume
    $f\not\in\ran(A)$. Because $\ran(A')$ is closed, $\ran(A)$ is also
     closed in view of Theorem \ref{th:2.2}. Therefore, we have
     $d=\dst_{W'}(f,\ran(A))>0$ and can apply Lemma
     \ref{la:bh2}. Consequently, there exists an $\eta\in
     (W')'$ such that
  \[
   \dual{\eta,f}_{(W')',W'}=1,\qquad 
   \dual{\eta,g}_{(W')',W'}=0 \quad (g\in \ran(A)). 
  \]   
Because $W$ is reflexive, there exists a $w\in W$ such that 
  \begin{equation}
   \dualW{f,w}=1,\qquad 
   \dualW{g,w}=0 \quad (g\in \ran(A)). 
\label{eq:165}   
  \end{equation}
By the second identity of \eqref{eq:165}, we have $0=\dualW{Av,w}=\dualV{v,A'w}$ for
     any $v\in V$, which implies that $A'w=0$. Therefore
      $w\in\nullsp(A')$. Because $f\in \nullsp(A')^\bot$,
      $\dualW{f,w}=0$. However, this contradicts to the first equality of
     \eqref{eq:165}. 
    \end{proof}

      \begin{theorem}[{\cite[Lemma 335]{kat58}}]
       \label{th:2.4}
       $\ran (A')\supset \nullsp(A)^\bot$ if $\ran(A)$ is closed. 
      \end{theorem}

   \begin{proof}
    Letting $f\in \nullsp(A)^\bot$, then
we introduce a linear functional $\phi_f$ on $R=\ran(A)$ by setting 
    $\phi_f(Av)=\dualV{f,v}$ for $v\in V$, which is
    possible because $\dualV{f,v}=0$ for $v\in \nullsp (A)$.
    The functional $\phi_f$ is bounded. In fact, we have 
    \[
     |\phi_f(Av)|=|\dualV{f,v}|\le \|f\|_{V'}\|v\|_V
    \]
and $v$ might be replaced by $v-g$ with any $g\in
    \nullsp(A)$. Consequently,
    \[
    |\phi_f(Av)|\le \|f\|_{V'}\|\tilde{v}\|_{\tilde{V}}\le
    \|f\|_{V'}\frac{1}{\gamma(A)}\|Av\|_{W'}
    \]
which implies that $\|\phi_f\|_{R'}=\sup_{\psi\in
    R}|\phi_f(\psi)|/\|\psi\|_{W'}\le \gamma(A)^{-1}\|f\|_{V'}$.
    By Hahn--Banach theorem, there exists a $\tilde{\phi}_f\in (W')'$ such
    that
\[
 \dual{\tilde{\phi}_f,\psi}_{(W')',W'}=\phi_f(\psi)\quad (\psi\in
    R),\qquad \|\tilde{\phi}_f\|_{(W')'}\le\gamma(A)^{-1}\|f\|_{V'}.
\]    
Because $W$ is reflexive, there exists a $w\in W$ such that
\[
 \dual{\tilde{\phi}_f,\psi}_{(W')',W'}=\dualW{\psi,w}\qquad (\forall
    \psi\in W').
\]
Summing up, we deduce    
    \begin{equation*}
    \dual{Av,w}_{W',W}=\phi_f(Av)=\dualV{f,v}\qquad (v\in V).
    \end{equation*}
This relation implies the expression $f=A'w$: $f\in \ran(A')$.
   \end{proof}

Now, we can prove the following
well-known result called Closed Range Theorem.

\begin{corollary}
 \label{th:2.10}
 The following \textup{(i)--(iv)} are equivalent: 
 \begin{itemize}
  \item[\textup{(i)}] $\ran(A)$ is closed; 
  \item[\textup{(ii)}] $\ran(A')$ is closed;
\item[\textup{(iii)}] $\ran (A)=\nullsp(A')^\bot$; 
\item[\textup{(iv)}] $\ran (A')=\nullsp(A)^\bot$.
 \end{itemize}
\end{corollary}

\begin{proof}
(i) $\Leftrightarrow$(ii): We have already verified this part. See Theorems
 \ref{th:2.1} and \ref{th:2.2}.
 
\noindent (iv)$\Rightarrow$(ii):   
If $\ran (A')=\nullsp(A)^\bot$, then $\ran(A')$ is closed because
 $\nullsp(A)^\bot$ is closed.
 
\noindent (ii)$\Rightarrow$(iv):
 Assuming that $\ran(A')$ is closed,
 then we can apply Lemma \ref{la:2.3} and Theorem \ref{th:2.4} to deduce
 $\ran (A')=\nullsp(A)^\bot$.

\noindent (i)$\Leftrightarrow$(iii): It is exactly the same as that of the
 part ``(ii)$\Leftrightarrow$(iv)''.  
\end{proof}

\begin{remark}
 \label{rem:31}
In the discussion presented above, the boundedness of $A$ plays no
 essential role. All the theorems and their proofs remain valid
 for a closed linear operator $A$ if the dual operator $A'$ is well-defined.  
\end{remark}

\begin{remark}
 \label{rem:32}
The original version of Closed Range Theorem could be found in
 \cite[Theorems X.8, X.9]{ban87}. 
\end{remark}

\begin{remark}
\label{rem:77} 
In this section, we considered $A\in\mathcal{L}(V,W')$ with the intention of applying results to the proof of Theorems \ref{th:bnb} and \ref{th:5.1}. However, if we consider a linear densely defined closed operator $T$ from a Banach space $X$ to a Banach space $Y$, we can prove the following results in exactly the same way. In particular, $Y$ needs not to be reflexive.  
\begin{itemize}
 \item $\ran(T)$ is closed if and only if $\gamma(T)>0$. 
 \item $\gamma(T)=\gamma(T')$. 
 \item The following \textup{(i)--(iv)} are equivalent: 
 \begin{itemize}
  \item[\textup{(i)}] $\ran(T)$ is closed; 
  \item[\textup{(ii)}] $\ran(T')$ is closed;
\item[\textup{(iii)}] $\ran (T)=\nullsp(T')^\bot$; 
\item[\textup{(iv)}] $\ran (T')=\nullsp(T)^\bot$.
\end{itemize}
\end{itemize} 
\end{remark}

\section{Proof of Theorems \ref{th:bnb} and \ref{th:5.1}}
\label{sec:5}

It suffices to state the proof of Theorem \ref{th:5.1} because 
Theorems \ref{th:bnb} and \ref{th:5.1} are equivalent through the
relation \eqref{eq:a1}.
 
 \begin{proof}[Proof of Theorem \ref{th:5.1}, the part 
\textup{``(i) $\Leftrightarrow$ (iii)''}] 
\mbox{ }
  \begin{itemize}
  \item[] $A$: bijective
  \item[]  $\Rightarrow$ $\nullsp (A)=\{0\}$, $\ran(A)=W'$
  \item[]  $\Rightarrow$ $\nullsp (A)=\{0\}$, $\ran(A)\mbox{: closed}$, 
    $\nullsp(A')^\bot=\ran(A)=W'$ \hfill (by Corollary \ref{th:2.10})
\item[] $\Rightarrow$ $\mu(A)=\gamma(A)$, $\gamma(A)>0$, $\nullsp (A')=\{0\}$
 \hfill (by \eqref{eq:2.11a}, Th \ref{th:2.1})
\item[] $\Rightarrow$ $\mu(A)=\gamma(A)$, $\gamma(A)>0$, $\mu(A')=\gamma(A')$
\hfill (by \eqref{eq:2.11a})
\item[] $\Rightarrow$ $\mu(A)=\mu(A')>0$ \hfill   (by Theorem \ref{th:2.2})
\item[] $\Rightarrow$ $\nullsp (A)=\{0\}$, $\nullsp (A')=\{0\}$, $\gamma(A)=\gamma(A')>0$
\hfill (by \eqref{eq:2.11b}, \eqref{eq:2.11a})
\item[] $\Rightarrow$ $\nullsp (A)=\{0\}$, $\ran(A)\mbox{: closed}$,
	$\ran(A)=\nullsp(A')^\bot$
\item[] \mbox{ }\hfill (by Theorems \ref{th:2.1},
	\ref{th:2.2} and Corollary \ref{th:2.10})
\item[] $\Rightarrow$ $A\mbox{: bijective}$.  
  \end{itemize}
 \end{proof}
 
  \begin{proof}[Proof of Theorem \ref{th:5.1}, the part 
\textup{``(ii) $\Leftrightarrow$ (iii)''}] 
\mbox{ }
   \begin{itemize}
  \item[] $\mu(A)=\mu(A')>0$ 
  \item[] $\Rightarrow$ $\mu(A)>0$, $\nullsp(A')=\{0\}$ \hfill   (by \eqref{eq:2.11b})
  \item[] $\Rightarrow$ $\mu(A)>0$, $\nullsp(A)=\{0\}$,
	  $\gamma(A')=\mu(A')$ \hfill (by \eqref{eq:2.11a}, \eqref{eq:2.11b})
\item[] $\Rightarrow$ $\mu(A)>0$, $\gamma(A)=\mu(A)$,
	$\gamma(A')=\mu(A')$ \hfill {(by \eqref{eq:2.11a})}
\item[] $\Rightarrow$ $\mu(A)=\mu(A')>0$.  
   \end{itemize}
  \end{proof}

 \section{Application to evolution equations of parabolic type} 
\label{sec:7}

In this section, we present an application of Theorem \ref{th:bnb} to
evolution equations of parabolic type.

\subsection{Example}
\label{sec:e1}
We start with a concrete example. Letting $J=(0,T)$ with $T>0$, and supposing that $\Omega$ is a
 Lipschitz domain in $\mathbb{R}^d$, $d\ge 1$, we consider the initial-boundary
 value problem
 \begin{subequations} 
  \label{eq:cd}
\begin{align}
 \partial_tu&=\nabla\cdot \nu(x,t) \nabla u-\nabla \cdot (\bm{b}(x,t) u)
 && \nonumber \\
 & \mbox{ }\qquad\qquad\qquad -c(x,t)u+F(x,t)&& (x\in\Omega,~t\in J),\label{eq:cd1}\\
 u&=0 && (x\in\partial\Omega,~t\in J),\label{eq:cd2}\\
 u(x,0)&=u_0(x)  && (x\in\Omega),\label{eq:cd3}
\end{align}
 \end{subequations}
 where $\nu,\bm{b},c,F$ and $u_0$ are given functions.

Several frameworks and methods are used to establish the well-posedness
(the unique existence of a solution with a priori estimate) of
\eqref{eq:cd}:
\begin{itemize}
\item Semigroup method (\cite{paz83} for example); 
\item Variational method: Galerkin method based on compactness theorems 
(\cite{dl92} and
\cite{wlo87} for example); 
\item Variational method: Operator method (\cite{lm72} for
      example). 
\end{itemize}

As described, we present another variational method. To this end, we
first derive a weak formulation of \eqref{eq:cd}. For the time being, those
$\nu,\bm{b},c,F$ and $u_0$ are assumed to be suitably smooth as well as
a solution $u$. Set
  \[
   \mathcal{D}=\{v=\tilde{v}|_{J\times\Omega}\mid \tilde{v}\in
   C^\infty(\mathbb{R}\times
   \mathbb{R}^d),~\operatorname{supp}\tilde{v}\subset J\times \Omega\}.
  \]

 Multiplying both sides of \eqref{eq:cd1} by
 $v\in \mathcal{D}$, integrating it in $x\in \Omega$
 and $t\in J$ and using the boundary condition \eqref{eq:cd2}, 
 we obtain
  \begin{multline}
   \int_J\int_\Omega (\partial_tu) v~dxdt+
   \int_J\int_\Omega \left[\nu(x,t)\nabla u\cdot\nabla
   v-\bm{b}(x,t)u\cdot \nabla v+c(x,t)uv\right]~dxdt\\
   =\int_J\int_\Omega Fv~dxdt.
  \label{eq:cd10}
  \end{multline}

 We introduce
\begin{align*}
 & H=L^2(\Omega),&& (\cdot,\cdot)=(\cdot,\cdot)_H=(\cdot,\cdot)_{L^2(\Omega)},&& \|\cdot\|_H=\|\cdot\|_{L^2(\Omega)},\\
 & V=H^1_0(\Omega),&& (\cdot,\cdot)_V=(\nabla \cdot,\nabla \cdot)_{L^2(\Omega)}, &&\|\cdot\|_V=\|\nabla
 \cdot\|_{L^2(\Omega)}
\end{align*}
and 
\[
 \dual{\cdot,\cdot}=\dualV{\cdot,\cdot}= \mbox{the duality pairing
 between $V$
 and $V'$}.
 \]
Moreover, set 
\begin{align*}
 a(t;w,v)&=\int_\Omega \left[\nu(x,t)\nabla w\cdot\nabla
 v-\bm{b}(x,t)w\cdot \nabla v+c(x,t)wv\right]~dx,\\
 \dual{f,v}&=\int_\Omega Fv~dx
\end{align*}
for $w,v\in V$.

We make the following assumptions:  
\begin{subequations} 
\label{eq:aa}
\begin{gather}
\exists \nu_0>0,\quad \nu(x,t)\ge \nu_0>0\quad (x\in\Omega,~t\in J),\qquad \nu\in L^\infty(\Omega);\label{eq:aa1} \\
\bm{b}\in L^\infty(\Omega\times J)^d,\quad c\in L^\infty(\Omega\times J);\label{eq:aa2}\\ 
\exists c_0>0,\quad \frac12 \nabla\cdot \bm{b}(x,t)+c(x,t)\ge c_0>0\quad (x\in\Omega,~t\in J).\label{eq:aa3}
\end{gather}
\end{subequations}
 
 Using \eqref{eq:aa1}, \eqref{eq:aa2}, \eqref{eq:aa3} and Poincar{\' e} inequality
\[
 \|v\|_V\le C_{\textrm{P}}\|v\|_{H^1(\Omega)}\qquad (v\in V),
\]
one can prove that there exist positive constants $M$ and $\alpha$ which
depend only on $\nu_0$, $c_0$, $\|\nu\|_{L^\infty(\Omega)}$, $\|\bm{b}\|_{L^\infty(\Omega)^d}$,
$\|c\|_{L^\infty(\Omega)}$ and $C_{\textrm{P}}$ such that 
\begin{align*}
|a(t;w,v)|&\le M\|w\|_V\|v\|_V && (w,v\in V,~ t\in J),\\
a(t;v,v)&\ge \alpha \|v\|_V^2 && (v\in V,~ t\in J).
\end{align*}
 
Therefore, for a.e.~$t\in J$, we can introduce a linear operator $A(t)$
from $V$ into $V'$ as 
\begin{equation}
 \label{eq:at}
  \dual{A(t)w,v}=a(t;w,v)\qquad (w,v\in V,~t\in J)
\end{equation}
satisfying
\begin{subequations}
\begin{align}
& \dual{A(t)w,v}\le M\|w\|_V\|v\|_V && (w,v\in V,~t\in J),\label{eq:0a}\\
& \dual{A(t)v,v}\ge \alpha \|v\|_V^2 && (v\in V,~t\in J).\label{eq:0b}
\end{align}
\label{eq:00}
\end{subequations}

As a result, \eqref{eq:cd10} is expressed as
\begin{equation}
 \label{eq:cd11}
  \int_J\dual{\partial_t
  u,v}~dt+\int_J\dual{A(t)u,v}~dt=\int_J\dual{f,v}~dt\qquad (v \in \mathcal{D}). 
\end{equation}
However, the initial condition \eqref{eq:cd3} is interpreted as
\begin{equation}
 \label{eq:cd12}
  (u(0),v)=(u_0,v)\qquad (v\in H).
\end{equation}

At this stage, we introduce the following function spaces: 
\begin{align*}
&\mathcal{X}=L^2(J;V)\cap H^1(J;V'),&&   \|u\|_{\mathcal{X}}^2=\|u\|_{L^2(J;V)}^2+\|u'\|_{L^2(J;V')}^2,\\
&\mathcal{Y}_1=L^2(J;V), &&\|v_1\|_{\mathcal{Y}_1}^2=\|v_1\|_{L^2(J;V)}\\
&\mathcal{Y}=\mathcal{Y}_1 \times H, && \|v\|_{\mathcal{Y}}^2 =\|v_1\|_{L^2(J;V)}^2+\|v_2\|_H^2,
\end{align*}
where
\begin{align*}
 L^2(J;V)&=\{v:J\to V \mid \|v\|_{L^2(J;V)}<\infty\},
  && \|v\|_{L^2(J;V)}^2=\int_J\|v\|^2_V~dt, \\
 H^1(J;V')&=\{v:J\to V \mid \|v\|_{H^1(J;V')}<\infty\},
  && \|v\|_{H^1(J;V')}^2=\int_J (\|v\|^2_{V'}+\|v'\|^2_{V'})~dt.
\end{align*}
It is noteworthy that $\mathcal{D}$ is dense in $\mathcal{Y}_1$. 

We can state the weak formulation of \eqref{eq:cd} as follows. Assuming
\begin{equation}
 f\in L^2(J;V'),\qquad u_0\in H,
 \label{eq:data}
\end{equation}
we find $u\in \mathcal{X}$ such that
\begin{multline}
\underbrace{\int_J \left[
 \dual{u',v_1}+\dual{A(t)u,v_1}\right]~dt+(u(0),v_2)}_{=B(u,v)}
 \\
 =\int_J\dual{f,v_1}~dt + (u_0,v_2)\qquad (\forall v=(v_1,v_2)\in\mathcal{Y}),
\label{eq:1a}
\end{multline}
where $u'$ denotes $du(t)/dt$. 
Alternatively, \eqref{eq:1a} is expressed formally as
\begin{equation}
u'+A(t)u=f(t),\quad t\in J;\qquad u(0)=u_0. 
\label{eq:0}
\end{equation}

 \begin{remark}
  In \eqref{eq:data}, $f\in L^2(J;V')$ is guaranteed by assuming $F\in
  L^2(J;H)$. Moreover, $u(0)\in H$ is well-defined; see Lemma
  \ref{rem:1}. 
 \end{remark}


\subsection{Problem}
\label{sec:e2}

We consider more general settings.
Letting $H$ and $V$ be (real) Hilbert spaces such that $V\subset H$ is dense
with the continuous injection, then the inner product and norms are denoted
as $(\cdot,\cdot)=(\cdot,\cdot)_H$, 
$(\cdot,\cdot)_V$, $\|\cdot\|=\|\cdot\|_H$ and $\|\cdot\|_V$. The topological dual spaces
$H$ and $V$ are denoted, respectively, by $H'$ and $V'$.
As usual, we identify $H$ with $H'$ and consider the triple $V\subset H\subset V'$. Moreover,
$\dual{\cdot,\cdot}=\dualV{\cdot,\cdot}$ denotes 
duality pairing between $V'$ and $V$. Consider function spaces
$\mathcal{X}$, $\mathcal{Y}_1$ and $\mathcal{Y}$ as presented above. 

Supposing that, for a.e.~$t\in J$, we are given a linear operator $A(t)$
of $V\to V'$ satisfying \eqref{eq:00}, where $M$ and $\alpha$ are
positive constants independent of $t\in J$.
Without loss of generality, we assume that $\alpha\le 1\le M$. 
Given \eqref{eq:data}, we
consider the abstract evolution equation of parabolic type
\eqref{eq:1a}.  

The following result is called the trace theorem (see \cite[theorem XVIII-1]{dl92}, \cite[theorem 25.2]{wlo87}, \cite[theorem 41.15]{zen90}). 

\begin{lemma}
There exists a positive constant $C_{\mathrm{Tr},T}$ depending only on $T$ such that  
  \begin{equation}
\max_{t\in \overline{J}}\|v(t)\|_H \le C_{\mathrm{Tr},T}\|v\|_{\mathcal{X}}\qquad
 (v\in \mathcal{X}).
  \label{eq:ch}
  \end{equation}
In other words, the space $\mathcal{X}$ is embedded continuously in the
 set of $H$-valued continuous functions on $\overline{J}$. Particularly, $u(0)\in H$ in \eqref{eq:1a} is
   well-defined.  
\label{rem:1} 
\end{lemma}

The main result of this section is the following result, which is often
called the Lions Theorem.

 \begin{theorem}
Given \eqref{eq:data}, problem \eqref{eq:1a} admits a unique solution  
$u\in\mathcal{X}$ that satisfies
\begin{equation}
 \|u\|_\mathcal{X}\le C\left(\|f\|_{L^2(J;V')}+\|u_0\|_H\right),
\label{eq:100}
\end{equation}
where $C$ denotes a positive constant depending only on $M$ and
$\alpha$.
 \label{th:lions}
 \end{theorem}

To prove this theorem, it suffices to verify the following: 
\begin{subequations} 
 \label{eq:101}
 \begin{gather}
\exists \mu>0,\quad
  \sup_{u\in\mathcal{X},v\in\mathcal{Y}}\frac{B(u,v)}{\|u\|_{\mathcal{X}}\|v\|_{\mathcal{Y}}}=
  \mu; \label{eq:101c}\\
\exists \beta>0,\quad \inf_{u\in\mathcal{X}}\sup_{v\in\mathcal{Y}}\frac{B(u,v)}{\|u\|_{\mathcal{X}}\|v\|_{\mathcal{Y}}}= \beta ; \label{eq:101a}\\
v\in \mathcal{Y},\quad (\forall u\in\mathcal{X},\ B(u,v)=0)\quad
 \Longrightarrow \quad (v=0). \label{eq:101b}
 \end{gather}
\end{subequations}
Subsequently, we can apply Theorem \ref{th:bnb} to conclude a unique existence
of the solution $u$. Moreover, the a priori estimate \eqref{eq:100} is a
readily obtainable consequence of \eqref{eq:101a}.

\subsection{Proof of Theorem \ref{th:lions}}

We use the following auxiliary results. 
By virtue of \eqref{eq:00}, $A(t)$ is invertible for a.e. $t\in J$. 
Moreover, we have the following. 

\begin{lemma}
 \label{la:5.1}
\textup{(i)} $\displaystyle{\|A(t)^{-1}g\|_V\le \frac{1}{\alpha}\|g\|_{V'}}$
 for all $g\in V'$ and a.e.~$t\in J$.  \\
\textup{(ii)} $\displaystyle{\dual{g,A(t)^{-1}g}\ge \frac{\alpha}{M^2}\|g\|_{V'}}$ for all
 $g\in V'$ and a.e.~$t\in J$. 
\end{lemma}

\begin{proof}
\noindent \textup{(i)} For $g\in V$, set $v=A(t)^{-1}g\in V$. Then,
 $\displaystyle{\dual{g,A(t)^{-1}g}=\dual{A(t)v,v}\ge \alpha
\|v\|_V^2}$. However, $|\dual{g,A(t)^{-1}g}|\le
 \|g\|_{V'}\|A(t)^{-1}g\|_V=\|g\|_{V'}\|v\|_V$. Combining these, we have
 $\|A(t)^{-1}g\|_V=\|v\|\le (1/\alpha)\|g\|_{V'}$. 

 \noindent \textup{(ii)} \eqref{eq:0a} implies $\|A(t)v\|_{V'}\le
 M\|v\|_V$ for $v\in V$. Now, set $v=A(t)^{-1}g\in V$ for $g\in
 V'$. Then, $\|g\|_{V'}= \sup_{w\in V}|\dual{g,w}|/\|w\|_V=\sup_{w\in
 V}|\dual{A(t)v,w}|/\|w\|_V\le M\|v\|$. Combining this with $\displaystyle{\dual{g,A(t)^{-1}g}\ge \alpha
 \|v\|_V^2}$, we obtain the desired inequality. 
\end{proof}

We introduce an alternate norm of $\mathcal{X}$ as
\[
 \vnorm{w}_{\mathcal{X}}^2=\int_J\|w'(t)+A(t)w(t)\|_{V'}^2~dt+\|w(0)\|^2
\]
for $w\in\mathcal{X}$. 

\begin{lemma}
 \label{la:5.20}
Two norms $\|\cdot\|_{\mathcal{X}}$ and $\vnorm{\cdot}_{\mathcal{X}}$
 are equivalent in $\mathcal{X}$. In particular, 
 \[
\alpha\|w\|_{\mathcal{X}}\le \vnorm{w}_{\mathcal{X}}\le C_{\max}\|w\|_{\mathcal{X}}
 \]
for $w\in \mathcal{X}$, where $C_{\max}^2=1+M^2+C_{\mathrm{Tr},T}^2$. 
 \end{lemma}

  \begin{proof}
   For $w\in\mathcal{X}$, we calculate as 
\begin{align*}
 \vnorm{w}_{\mathcal{X}}^2
 &=\int_J\left[\sup_{v\in V}
 \frac{\dual{w'+Aw,v}}{\|v\|_V}\right]^2dt+\|w(0)\|^2\\
 &= \int_J\left[\|w'\|_{V'}+\sup_{v\in V}
 \frac{\dual{Aw,v}}{\|v\|_V}\right]^2dt+\|w(0)\|^2\\
 &\le\int_J\left[\|w'\|_{V'}+M\|w\|_V\right]^2dt+C_{\mathrm{Tr},T}^2\|w\|_{\mathcal{X}}^2\\ 
 &\le (1+M^2+C_{\mathrm{Tr},T}^2)\|w\|_{\mathcal{X}}^2 
\end{align*}
   and
\[
  \vnorm{w}_{\mathcal{X}}^2
 \ge \int_J\left[\|w'\|_{V'}+
 \frac{\dual{Aw,w}}{\|w\|_V}\right]^2dt
 \ge\int_J\left[\|w'\|_{V'}+\alpha\|w\|_V\right]^2dt 
 \ge \alpha^2\|w\|_{\mathcal{X}}^2. 
\]
\end{proof}

The following lemma can be found in 
\cite[Theorem 2, \S XVIII-1]{dl92} and 
\cite[Theorem 41.15]{zen90}. 

\begin{lemma}
 \label{la:5.2}
 For $w,v\in \mathcal{X}$, we have
\begin{subequations} 
 \begin{equation}
 \label{eq:105}
  \int_J \dual{w',v}~dt=(w(T),v(T))-(w(0),v(0))-\int_J\dual{v',w}~dt
 \end{equation} 
 and
  \begin{equation}
   \int_J\dual{w',w}~dt=\frac{1}{2}\left(\|w(T)\|^2-\|w(0)\|^2\right)\ge
    -\frac12 \|w(0)\|^2.
 \label{eq:106}
  \end{equation}
\end{subequations}
\end{lemma}

Now we can state the following proof.

  \begin{proof}[Proof of \eqref{eq:101c}]
We apply the Cauchy--Schwarz inequality and Lemma
 \ref{la:5.20} to obtain  
\begin{align*}
 B(u,v)
 & =\int_J\left[\dual{u',v_1}+\dual{Au,v_1} \right]~dt +(u(0),v_2) \\
 & =\int_J\dual{u'+Au,v_1}~dt +(u(0),v_2) \\
 &
 \le \vnorm{u}_{\mathcal{X}}\|v\|_{\mathcal{Y}}\le C_{\max}\|u\|_{\mathcal{X}}\|v\|_{\mathcal{Y}}
\end{align*}
for $u\in \mathcal{X}$ and $v=(v_1,v_2)\in \mathcal{Y}$.
\end{proof}

\begin{proof}[Proof of \eqref{eq:101a}]
Let $u\in \mathcal{X}$ be arbitrary. Set $v_1=A(t)^{-1}u'+u\in L^2(J;V)$,
 $v_2=u(0)\in H$ and $v=(v_1,v_2)\in\mathcal{Y}$. Using Lemma
 \ref{la:5.1}, we have
 \begin{align*}
   \|v\|_{\mathcal{Y}}^2
 &=\int_J\|A^{-1}u'+u\|_{V}^2~dt+\|u(0)\|^2\\
 &=\int_J\|A^{-1}(u'+Au)\|_{V}^2~dt+\|u(0)\|^2\\
 &\le\frac{1}{\alpha^2}\int_J\|u'+Au\|_{V}^2~dt+\|u(0)\|^2 \le
 \frac{1}{\alpha^2} \vnorm{u}_{\mathcal{X}}^2.
 \end{align*}
 Moreover,
 \begin{align*}
  B(u,v)
 &=\int_J\dual{u'+Au,A^{-1}u'+u}~dt +(u(0),u(0))\\
 &=\int_J\dual{u'+Au,{A^{-1}}(u'+Au)}~dt +\|u(0)\|^2\\
 &\ge \frac{\alpha}{M^2} \int_J\|u'+Au\|_{V'}^2~dt +\|u(0)\|^2\\
 &\ge \frac{\alpha}{M^2} \vnorm{u}_{\mathcal{X}}^2\ge
 \frac{\alpha^2}{M^2} \vnorm{u}_{\mathcal{X}}\|v\|_{\mathcal{Y}}.
 \end{align*}
Consequently, using Lemma \ref{la:5.20}, we obtain 
 \[
  B(u,v)\ge \frac{\alpha^3}{M^2}\|u\|_{\mathcal{X}}\|v\|_{\mathcal{Y}},
  \]
 which implies \eqref{eq:101a}. 
\end{proof}

\begin{proof}[Proof of \eqref{eq:101b}]
Assume that $v=(v_1,v_2)\in\mathcal{Y}$ satisfies $B(u,v)=0$ for all
 $u\in \mathcal{X}$. That is, we assume that  
 \begin{equation}
  \int_J[\dual{u',v_1}+\dual{Au,v_1}]~dt+(u(0),v_2)=0\qquad (\forall
   u\in \mathcal{X}).
  \label{eq:151}
 \end{equation}
For any $\ep>0$, we take $u^*\in\mathcal{X}$ such that $u(0)=v_2$ and
 $u(t)=0$ for $t\ge \ep$. Substituting $u=u^*$ for \eqref{eq:151}, we
 have
 \[
 \int_0^\ep[\dual{{u^*}',v_1}+\dual{Au^*,v_1}]~dt+\|v_2\|^2=0.
 \]
Because $\ep$ is arbitrarily chosen, we infer that $v_2=0$.
Moreover, we have $v_1\in H^1(J;V')$. In fact, letting
 $u=\tilde{u}\phi\in\mathcal{X}$ 
 with $\tilde{u}\in V$ and $\phi\in C_0^\infty(J;\mathbb{R})$, we have
\[
 \int_J\dual{\phi'\tilde{u},v_1}~dt =-\int_J \dual{A\phi\tilde{u},v_1}~dt.
\]
This result implies that 
 \[
\left\langle \int_0^T v_1 \phi'~dt,\tilde{u}\right\rangle =\left\langle -\int_0^T A'v_1\phi~dt
 ,\tilde{u}\right\rangle
 \]
Consequently, we deduce $v_1'\in L^2(J;V')$ and $v_1'=A(t)'v_1$.

 Therefore, we can apply \eqref{eq:105} to obtain
  \begin{equation*}
  \int_J[-\dual{v_1',\psi}+\dual{A'v_1,\psi}]~dt=0\qquad (\forall
   \psi\in C_0^\infty(J;V)).
 \end{equation*}
 Because $C_0^\infty(J;V)$ is dense in ${L^2(J;V)}$, this gives
   \begin{equation}
  \int_J[-\dual{v_1',w}+\dual{A'v_1,w}]~dt=0\qquad (\forall
   w\in L^2(J;V)).
  \label{eq:152}
   \end{equation}
Letting $\tilde{w}\in V$ arbitrarily and substituting $w=t\tilde{w}$ for \eqref{eq:152}, we
 have
 \[
  \int_J[-\dual{v_1',t\tilde{w}}+\dual{A'v_1,t\tilde{w}}]~dt=0.
 \]
 Again we apply \eqref{eq:105} to obtain
\[
 -(T\tilde{w},v_1(T))+\int_J\dual{(t\tilde{w})',v_1}+\int_J \dual{A(t)(t\tilde{w}),v_1}~dt=0.
\]
Choosing $\tilde{w}=v_1(T)$ and using \eqref{eq:151}, we obtain $v_1(T)=0$.

 At this stage, substituting $w=v_1$ for \eqref{eq:152} and using
 \eqref{eq:0b} and \eqref{eq:106}, then we have
\[
 -\frac{1}{2}\|v_1(T)\|^2+\frac{1}{2}\|v_1(0)\|^2+\alpha
 \int_J\|v_1\|_V^2~dt\le 0.
\]
This result implies that $v_1=0$, which completes the proof.
\end{proof}

\begin{remark}
The case $u_0=0$ is described explicitly
in \cite{eg04}.
\end{remark}

\appendix

\section{Proof of ``\eqref{eq:pf10} $\Rightarrow$ \eqref{eq:pf11}''}
\label{sec:a}

We prove a more general lemma described below.    

\begin{lemma}
Let $T$ be a linear closed operator of a Banach space $X$ to a
 (possibly another) Banach
 space $Y$. Then,
\begin{equation}
 \label{eq:a101}
  \ball{Y}{r}\subset \overline{T\ball{X}{1}}\mbox{ with }r>0 
\end{equation}
implies that 
 \begin{equation}
 \label{eq:a102}
 \ball{Y}{r}\subset T\ball{X}{1}.
 \end{equation}
\label{la:aaa}
\end{lemma}

Recall that $\ball{Y}{r}=\{y\in Y\mid
    \|y\|_Y<r\}$ and $\overline{T\ball{X}{1}}$ denotes the closure
    of $T\ball{X}{1}=\{Tx\in Y\mid x\in \ball{X}{1}\}$ in $Y$. 
    To show the lemma, we apply a standard argument usually used to prove Open Mapping Theorem or Closed Graph Theorem.   
    
\begin{proof}
Assume that \eqref{eq:a101} is satisfied.  
Let $\sigma>0$ be arbitrary. For the time being, we admit that   
\begin{equation}
 \label{eq:a109}
\ball{Y}{r}\subset T\ball{X}{1+\sigma}. 
\end{equation}
Then, for any $0<r'<r$, choosing $\sigma=r/r'-1>0$, we
 have 
\[
  \ball{Y}{r'}
 =\frac{r'}{r}\ball{Y}{r}
 \subset \frac{r'}{r}T\ball{X}{1+\sigma}=T\ball{X}{1}.
\]
The relation \eqref{eq:a102} is a readily obtainable consequence of this relation.  

We now verify that \eqref{eq:a109} is true; we will show that, for any $y\in \ball{Y}{r}$, there exists an
 $x\in \ball{X}{1+\sigma}$ satisfying $Tx=y$.
 
As just remarked above, \eqref{eq:a101} gives
 \begin{equation}
 \label{eq:a106}
 \ball{Y}{\lambda r}\subset \overline{T\ball{X}{\lambda}} 
\end{equation}
for any $\lambda>0$.


 Set $\ep=\sigma/(2+\sigma)<1$.  
According to \eqref{eq:a101}, there is a $y_0\in T\ball{X}{1}$
   satisfying $\|y-y_0\|_Y< \ep r$. That is, there is a $\xi_0\in \ball{X}{1}$ satisfying
  \[
   \|y-T\xi_0\|_Y< \ep r.
  \]
Then, we apply \eqref{eq:a106} with $\lambda=\ep$. Because
   $y-T\xi_0\in\ball{Y}{\ep r}$, there is a $\xi_1\in
  \ball{X}{\ep}$ satisfying
\[
   \|y-T\xi_0-T\xi_1\|_X< \ep^2r.
\]
Proceeding in this way, we can construct a sequence $\{\xi_n\}_{n\ge 0}$
   in $X$ with the properties
\[
 \|y-T\xi_0-T\xi_1-\cdots-T\xi_n\|_Y< \ep^{n+1}r,\quad \|\xi_n\|_X< \ep^n.
\]  
If we set $x_n=\xi_0+\xi_1+\cdots+\xi_n$, we have
\[
   \|x_{n+m}-x_{n}\|_X
   \le
   \sum_{j=n+1}^{n+m}\|\xi_{j}\|_X\le \sum_{j=n+1}^{n+m}\ep^{j} \le
   \frac{\ep^{n+1}}{1-\ep}\to 0\quad (n\to\infty).
\]
Therefore, there exists an $x\in X$ satisfying $x_n\to x$ in $X$ as
   $n\to\infty$. Moreover, we have $\|y-Tx_n\|_Y<\ep^{n+1}r\to 0$ as
   $n\to\infty$. This implies that $Tx=y$ because $T$ is
   closed. Finally,
 \[
  \|x\|_X\le \sum_{n=0}^\infty\|\xi_n\|_X\le
   \sum_{n=0}^\infty\ep^n
   =\frac{1}{1-\ep}=1+\frac{\sigma}{2}<1+\sigma;
 \]
therefore, we have $x\in\ball{X}{1+\sigma}$. This completes the proof of
 \eqref{eq:a109}. 
   
  \end{proof}

\section{Comments on the revised version}
\label{sec:b}

\begin{enumerate}
 \item Open Mapping Theorem recalled in the original version (November
       5, 2017) has been
       removed.
\item Lemmas \ref{la:bh1} and \ref{la:aaa} have been added; they are used in the proof of
      Theorem \ref{th:2.2}. 
 \item Proof of Theorem \ref{th:2.2} has been
       corrected. Consequently, the theorems and their proofs in \S \ref{sec:2} remain valid for a closed linear operator $A$ if the
       dual operator $A'$ is well-defined. See Remark \ref{rem:31}. 
\item Remark \ref{rem:77} has been added.  
 \item Proof of \eqref{eq:101a} has been modified. I believe that it is
       a new proof.   
\end{enumerate}

\section*{Acknowledgement}
I would like to thank Professor Gerd Wachsmuth for pointing out that the
original proof of Theorem \ref{th:2.2} was incomplete. This work was supported by JST CREST Grant Number JPMJCR15D1, Japan and
by JSPS KAKENHI Grant Number 15H03635, Japan.

\bibliographystyle{plain}

\begin{thebibliography}{10}

\bibitem{bab71}
I.~Babu{\v s}ka.
\newblock Error-bounds for finite element method.
\newblock {\em Numer. Math.}, 16:322--333, 1970/1971.

\bibitem{ba72}
I.~Babu{\v s}ka and A.~K. Aziz.
\newblock Survey lectures on the mathematical foundations of the finite element
  method.
\newblock In {\em The mathematical foundations of the finite element method
  with applications to partial differential equations ({P}roc. {S}ympos.,
  {U}niv. {M}aryland, {B}altimore, {M}d., 1972)}, pages 1--359. Academic Press,
  New York, 1972.
\newblock With the collaboration of G. Fix and R. B. Kellogg.

\bibitem{ban87}
S.~Banach.
\newblock {\em Theory of linear operations}.
\newblock North-Holland, 1987.
\newblock Translated from the 1979 French version by F. Jellett. The original
  Polish version was published in 1931.

\bibitem{bbf13}
D.~Boffi, F.~Brezzi, and M.~Fortin.
\newblock {\em Mixed finite element methods and applications}, volume~44 of
  {\em Springer Series in Computational Mathematics}.
\newblock Springer, Heidelberg, 2013.

\bibitem{bre11}
H.~Brezis.
\newblock {\em Functional analysis, {S}obolev spaces and partial differential
  equations}.
\newblock Universitext. Springer, New York, 2011.

\bibitem{bre74}
F.~Brezzi.
\newblock On the existence, uniqueness and approximation of saddle-point
  problems arising from {L}agrangian multipliers.
\newblock {\em Rev. Fran\c{c}aise Automat. Informat. Recherche Op\'erationnelle
  S\'er. Rouge}, 8({\rm R}-2):129--151, 1974.

\bibitem{bf91}
F.~Brezzi and M.~Fortin.
\newblock {\em Mixed and hybrid finite element methods}, volume~15 of {\em
  Springer Series in Computational Mathematics}.
\newblock Springer-Verlag, New York, 1991.

\bibitem{dl92}
R.~Dautray and J.~L. Lions.
\newblock {\em Mathematical analysis and numerical methods for science and
  technology. {V}ol. 5}.
\newblock Springer-Verlag, Berlin, 1992.
\newblock Evolution problems. I, With the collaboration of Michel Artola,
  Michel Cessenat and H\'el\`ene Lanchon, Translated from the French by Alan
  Craig.

\bibitem{eg02}
A.~Ern and J.~L. Guermond.
\newblock {\em {\' E}l{\' e}ments finis: th{\' e}orie, applications, mise en
  {\oe}uvre}, volume~36 of {\em Math{\' e}matiques \& Applications (Berlin)
  [Mathematics \& Applications]}.
\newblock Springer-Verlag, Berlin, 2002.

\bibitem{eg04}
A.~Ern and J.~L. Guermond.
\newblock {\em Theory and practice of finite elements}, volume 159 of {\em
  Applied Mathematical Sciences}.
\newblock Springer-Verlag, New York, 2004.

\bibitem{gt62}
H.~A. Gindler and A.~E. Taylor.
\newblock The minimum modulus of a linear operator and its use in spectral
  theory.
\newblock {\em Studia Math.}, 22:15--41, 1962/1963.

\bibitem{hay68}
T.~L. Hayden.
\newblock Representation theorems in reflexive {B}anach spaces.
\newblock {\em Math. Z.}, 104:405--406, 1968.

\bibitem{kat58}
T.~Kato.
\newblock Perturbation theory for nullity, deficiency and other quantities of
  linear operators.
\newblock {\em J. Analyse Math.}, 6:261--322, 1958.

\bibitem{kat95}
T.~Kato.
\newblock {\em Perturbation theory for linear operators}.
\newblock Classics in Mathematics. Springer-Verlag, Berlin, 1995.
\newblock Reprint of the 1980 edition.

\bibitem{lm54}
P.~D. Lax and A.~N. Milgram.
\newblock Parabolic equations.
\newblock In {\em Contributions to the theory of partial differential
  equations}, Annals of Mathematics Studies, no. 33, pages 167--190. Princeton
  University Press, Princeton, N. J., 1954.

\bibitem{lm72}
J.~L. Lions and E.~Magenes.
\newblock {\em Non-homogeneous boundary value problems and applications. {V}ol.
  {I}}.
\newblock Springer-Verlag, New York-Heidelberg, 1972.
\newblock Translated from the French by P. Kenneth, Die Grundlehren der
  mathematischen Wissenschaften, Band 181.

\bibitem{nec62}
J.~Ne{\v c}as.
\newblock Sur une m\'ethode pour r\'esoudre les \'equations aux d\'eriv\'ees
  partielles du type elliptique, voisine de la variationnelle.
\newblock {\em Ann. Scuola Norm. Sup. Pisa (3)}, 16:305--326, 1962.

\bibitem{nec67}
J.~Ne{\v c}as.
\newblock {\em Les m{\' e}thodes directes en th{\' e}orie des {\' e}quations
  elliptiques}.
\newblock Masson et Cie, {\' E}diteurs, Paris; Academia, {\' E}diteurs, Prague,
  1967.

\bibitem{nec12}
J.~Ne{\v{c}}as.
\newblock {\em Direct methods in the theory of elliptic equations}.
\newblock Springer Monographs in Mathematics. Springer, Heidelberg, 2012.
\newblock Translated from the 1967 French original by Gerard Tronel and Alois
  Kufner, Editorial coordination and preface by {\v{S}}{\'a}rka
  Ne{\v{c}}asov{\'a} and a contribution by Christian G. Simader.

\bibitem{nir55}
L.~Nirenberg.
\newblock Remarks on strongly elliptic partial differential equations.
\newblock {\em Comm. Pure Appl. Math.}, 8:649--675, 1955.

\bibitem{od96}
J.~T. Oden and L.~F. Demkowicz.
\newblock {\em Applied functional analysis}.
\newblock CRC Series in Computational Mechanics and Applied Analysis. CRC
  Press, Boca Raton, FL, 1996.

\bibitem{paz83}
A.~Pazy.
\newblock {\em Semigroups of linear operators and applications to partial
  differential equations}, volume~44 of {\em Applied Mathematical Sciences}.
\newblock Springer-Verlag, New York, 1983.

\bibitem{ros89}
I.~Ro{\c s}ca.
\newblock On the {B}abu\v ska--{L}ax--{M}ilgram theorem.
\newblock {\em An. Univ. Bucure\c sti Mat.}, 38(3):61--65, 1989.

\bibitem{sai17}
N.~Saito.
\newblock Variational analysis of the discontinuous {G}alerkin time-stepping
  method for parabolic equations.
\newblock arXiv:1710.10543.

\bibitem{sim72}
C.~G. Simader.
\newblock {\em On {D}irichlet's boundary value problem}.
\newblock Lecture Notes in Mathematics, Vol. 268. Springer-Verlag, Berlin-New
  York, 1972.
\newblock An $L^{p}$-theory based on a generalization of Gȧrding's inequality.

\bibitem{vis51}
M.~I. Vi{\v s}ik.
\newblock On strongly elliptic systems of differential equations.
\newblock {\em Mat. Sbornik N.S.}, 29(71):615--676, 1951.

\bibitem{wlo87}
J.~Wloka.
\newblock {\em Partial differential equations}.
\newblock Cambridge University Press, Cambridge, 1987.
\newblock Translated from the German by C. B. Thomas and M. J. Thomas.

\bibitem{zen90}
A.~{\v Z}en{\'\i}{\v s}ek.
\newblock {\em Nonlinear elliptic and evolution problems and their finite
  element approximations}.
\newblock Computational Mathematics and Applications. Academic Press, Inc.
  [Harcourt Brace Jovanovich, Publishers], London, 1990.
\newblock With a foreword by P.-A. Raviart.

\end{thebibliography}

\end{document}